\documentclass{article}
\usepackage{amsthm,amsmath,amssymb, pgf, color, tikz, pgfpict2e, graphicx, arydshln, subcaption,soul,enumitem}
\usepackage{tocloft}
\usepackage{comment}
\definecolor{aqua}{rgb}{0.0, 1.0, 1.0}
\definecolor{blue-violet}{rgb}{0.54, 0.17, 0.89}
\theoremstyle{plain}
\newtheorem{thm}{Theorem}[section]

\newtheorem{obs}[thm]{Observation}
\newtheorem{lem}[thm]{Lemma}
\newtheorem{prop}[thm]{Proposition}
\newtheorem{defn}[thm]{Definition}
\newtheorem{remark}[thm]{Remark}
\newtheorem{exa}[thm]{Example}

\tikzstyle{vertex}=[circle, draw, inner sep=0pt, minimum size=8pt]
\tikzstyle{bvertex}=[circle, blue!90!black, fill, draw=black, inner sep=0pt, minimum size=8pt]
\tikzstyle{rvertex}=[circle, red!50, fill, draw=black, inner sep=0pt, minimum size=8pt]
\tikzstyle{gvertex}=[circle, green!70!black, fill, draw=black, inner sep=0pt, minimum size=8pt]
\tikzstyle{Bvertex}=[circle, black, fill, draw, inner sep=0pt, minimum size=10pt]
\usetikzlibrary{decorations.pathmorphing}
\tikzset{
  wiggly/.style={
    decorate,
    decoration={snake, amplitude=.5mm, segment length=3.5mm}
  }
}

\newcommand{\bvertex}{\node[bvertex]}
\newcommand{\rvertex}{\node[rvertex]}
\newcommand{\gvertex}{\node[gvertex]}

\newcommand{\wgt}{$wgt$}
\newcommand{\ghat}{\widehat{G}}

\usepackage{authblk}

\newcommand{\CSB}{\mathcal{C}\mathcal{S}\mathcal{B}}
\newcommand{\OSB}{\mathcal{O}\mathcal{S}\mathcal{B}}

\newcommand{\SBV}{\mathcal{S}\mathcal{B}\mathcal{V}}
\newcommand{\BV}{\mathcal{B}\mathcal{V}}
\newcommand{\PB}{\mathcal{P}\mathcal{B}}
\newcommand{\NBC}{\mathcal{N}\mathcal{B}\mathcal{C}}
\newcommand{\CNBC}{\mathcal{C}\mathcal{N}\mathcal{B}\mathcal{C}}

\title{

Color $2$-switches and neighborhood  $\lambda$-balanced graphs with $k$ colors

}

\author[1]{Karen L. Collins} 
\affil[1]{Dept. of Math. and Comp. Sci, Wesleyan University}

\author[2]{Jonelle Hook} 
\affil[2]{Dept. of Math. and Comp. Sci., Mount St.\ Mary's University, Emmitsburg, MD}

\author[3]{Cayla McBee} 
\affil[3]{Dept. of Math. and Comp. Sci., Providence College, Providence, RI}

\author[4]{Ann N. Trenk} 
\affil[4]{Dept. of Math. and Stat., Wellesley College, Wellesley, MA}

\date{\today}

\begin{document}

\maketitle

\begin{abstract} 

This paper examines vertex colorings of graphs (not necessarily proper colorings) with constraints on the distribution of colors in vertex neighborhoods.  We introduce color 2-switches and color degree matrices.  The color degree matrix of a $k$-colored graph is an analog of the degree sequence, while a color 2-switch provides a way to transform a $k$-colored graph to another such graph while maintaining the color of each vertex and the multiset of colors in each vertex neighborhood.  We prove that two $k$-colored graphs have the same color degree matrix if and only if one can be obtained from the other by a sequence of color 2-switches.

In related work, we generalize the neighborhood balanced colorings in \cite{Coetal} and \cite{FM24} by allowing for $k$ colors (instead of two) and more flexibility on the number of vertices of each color in a neighborhood.  We introduce three classes of $k$-colored, $\lambda$-balanced graphs, in which any two color classes in a vertex neighborhood differ in size by at most $\lambda$.  These classes are distinguished by whether the balancing condition is imposed on the open neighborhood $N(v)$, the closed neighborhood $N[v]$, or allowed to vary by vertex.  For each class, the minimum $\lambda$ for which a graph admits a balanced coloring defines its $\lambda$-balance number.  We prove general results about these classes and their $\lambda$-balance numbers. For $k = 2$, we introduce a fourth class, parity balanced graphs, in which  the number of vertices of each color are equal in open neighborhoods for even-degree vertices  and in closed neighborhoods for odd-degree vertices. 

Additionally, we focus on the important case where $k=2$ and $\lambda \le 1$ and introduce the technique of red-blue removals.  We provide separating examples between these four classes and prove balance number results for paths, cycles, wheels, trees,  caterpillars, and complete multipartite graphs, and a counting result for caterpillars.

\end{abstract}

\section{Introduction}
\label{sec-intro}

Many classic graph properties can be described using local conditions.  For example, in defining a proper coloring, we can focus on one vertex at a time  to ensure that its color differs from that of each of its neighbors.  In this paper, we study colorings that are not necessarily proper, and we constrain the multisets of colors that appear in vertex neighborhoods.  If $G$ is a  graph and $v \in V(G)$, then the \emph{open neighborhood} of $v$, denoted by $N(v)$, is the set $\{u:uv \in E(G)\}$.  The \emph{closed neighborhood} of $v$, denoted by $N[v]$, is the set $ N(v) \cup \{v\}$.

Freyberg and Marr \cite{FM24} define a \emph{neighborhood balanced coloring} of a graph to be a coloring of the vertices using two colors (red and blue) so that  for each vertex $v$ the number of red and blue vertices in $N(v)$ is equal.  A graph with such a coloring is called an \emph{NBC} graph and the set of all NBC graphs is denoted by $\NBC$.  Collins et al. \cite{Coetal} define the analogous \emph{closed neighborhood balanced coloring}  when $N(v)$ is replaced by $N[v]$.  A graph with such a coloring is called a \emph{CNBC} graph and the set of all CNBC graphs is denoted by $\CNBC$. Recent work explores several generalizations of results in  \cite{Coetal} and \cite{FM24}.  In  \cite{MS_arxiv}, the authors generalize NBC graphs for three colors instead of two, and in \cite{APGS_arxiv} and  \cite{ASGP_arxiv}, the authors generalize NBC graphs and CNBC graphs for $k$ colors.  
Finally, in \cite{A_arxiv}, the author defines and studies quasi neighborhood balanced colorings which in the language of this paper are $(2,1)$-balanced colorings (see Definition~\ref{open-bal-def}).

By definition, every vertex of an NBC graph has even degree and every vertex of a CNBC graph has odd degree. 
In this paper, we relax the criteria for balance in several different ways to accommodate broader classes of graphs. In particular, we extend the framework to allow more than two colors and weaken the strict requirements for balance. These generalized classes model scenarios where multiple types of resources must be distributed among vertices such that each vertex has a balanced mix of resources for its neighborhood. For example, in farming or experimental design, one may wish to plant different types of crops in a layout so that each plot has a balance of crops in adjacent plots. 
 
A \emph{$k$-coloring} of a graph $G$ is a partition of $V(G)$ 
   into  color classes $C_1 \cup C_2 \cup \cdots \cup C_k$.  Vertex $x$ has \emph{color $i$} if $x\in C_i$, and in this case we write $c(x) = i$.
 Note that  in our $k$-colorings, it is permissible for adjacent vertices to receive the same color. The \emph{color $j$ degree}  (or $C_j$-degree) of vertex $x$ is   $|N(x) \cap C_j|$ and is denoted by $C_j$-$\deg(x)$.  With this new notation, a graph $G$ is an NBC graph if it has a $2$-coloring $V(G) = C_1 \cup C_2$ so that $C_1$-$\deg(x) = C_2$-$\deg(x)$  for all $x \in V(G)$. We define similar concepts for closed neighborhoods in Section \ref{sec-three-classes}.

In our next definition, we generalize the concept of the  degree sequence  to a graph with a $k$-coloring.  
For each vertex, we record  its color and the number of neighbors it has in each color class.  We store this information in a matrix where the rows are indexed by the vertices, the first $k$ columns are indexed by the colors 1 through $k$, and the last column is a \emph{color-identifier} column where we record the color of  each vertex. We illustrate the definition of a color degree matrix in  Example~\ref{exa:matrix}.

 \begin{defn} \rm 
Let $G$ be a graph with  vertex set   $\{v_1,v_2, \ldots, v_n\}$ and  $k$-coloring  $C_1 \cup C_2 \cup \cdots \cup C_k$. The \textit{color degree matrix} $D$  is the 
  $n \times (k+1)$ matrix in which $D_{ij} = C_j$-$\deg(v_i)$ for $j \leq k$ and $D_{ij} = c(v_i)$ for $j=k+1$. When there is more than one graph under consideration, we  denote the color degree matrix  of graph $G$ as $D(G)$.
  When $k \le 3$, we use red for color 1, blue  for color 2 and green for color 3, and in the color-identifier column we use {\text{\color{red}R}}, {\text{\color{blue}B}}, and {\text{\color{green!80!black}G}} in place of $1$, $2$, and $3$ respectively.
 \label{deg-matrix}
 \end{defn}

 \begin{remark} {\rm 
 A color degree matrix for a graph with a $k$-coloring is determined by an ordering of its vertices.  We write $D(G)$ for the color degree matrix of $G$ just as it is standard to write $A(G)$ for the adjacency matrix of $G$, which also depends on an ordering of the vertices of $G$. }
 \end{remark}

When the colors are permuted in a coloring of a graph, the color degree matrix changes.  For example, if the red and blue vertices are interchanged in the coloring of graph $G$ in Figure~\ref{house-fig}, then the resulting color degree matrix $D$ is obtained from $D(G)$ depicted in Figure \ref{fig:colordegreematrix} by swapping columns 1 and 2 and swapping {\text{\color{blue}B}}'s and {\text{\color{red}R}}'s in column 4.

\begin{exa} \label{exa:matrix}\rm 
Let $G, H, G'$, and $H'$ be the graphs shown in Figure~\ref{house-fig}. 
 Suppose there are three colors available: color 1 (red),  color 2 (blue), and color 3 (green) and  the vertices are colored  and indexed as in the figure.
The resulting color degree matrices are given in Figure~\ref{fig:colordegreematrix}.  Since there are no green vertices in these colorings, column $3$ consists of zeros in each of the matrices and no {\text{\color{green!80!black}G}} appears in the color-identifier column.
    
\end{exa}

Well-known theorems in graph theory characterize when a sequence of non-negative integers is graphic  and bigraphic  (e.g., see \cite[p. 45]{We01} and (\cite[p. 185]{We01}).
  This leads to an analogous result of characterizing which matrices are color degree matrices of graphs with a $k$-coloring, which we present in Observation~\ref{graphic-obs}.

\begin{obs}\rm
Let $D$ be an $n \times (k+1)$ matrix whose entries are nonnegative integers and refer to the entry $D_{i,k+1} $ as the color-identifier for row $i$. Then $D$ is the color degree matrix for a $k$-colored graph if and only if

\begin{enumerate} 
\item the color-identifier entry $D_{i,k+1}$ is in the set $ \{1,2,3, \ldots, k\}$ for each $i$, 

\item  the sequence of entries in column $j$ of $D$, whose row has color-identifier $j$, is a graphic sequence for each $j: 1\leq j\leq k$, and 

\item the pair $(p,q)$, where $p$ is the sequence  of   entries in column $j$ whose  row has color-identifier $i$, and  $q$ is the sequence of the entries in column $i$ whose row has  color-identifier  $j$, form a bigraphic sequence, for  each $i \neq j$.
\end{enumerate} 
\label{graphic-obs}
\end{obs}

In Figure \ref{fig:colordegreematrix}, the sequence of entries in column 1 of $D(G)$ corresponding to rows with color-identifier {\text{\color{red}R}} is the graphic sequence 1,1. Likewise, the sequence of entries in column 2 of $D(G)$ corresponding to rows with color-identifier {\text{\color{blue}B}} is the graphic sequence 0,1,1. Together, the sequences 2,1,1 and 2,2 constitute a bigraphic sequence.

The rest of the paper is organized as follows.
 In Section~\ref{sec-two-switch}, we define color $2$-switches which allow us to transform a $k$-colored graph to another $k$-colored graph while maintaining vertex colors and the quantity $|N(x) \cap C_j|$ for all vertices $x$ and all color classes $C_j$.  We prove that two graphs have the same color degree  matrix if and only if there is a sequence of color $2$-switches that transforms one graph to the other.
 In Section~\ref{sec-three-classes}, we define three classes of $k$-colored, $\lambda$-balanced graphs that generalize NBC and CNBC graphs allowing for $k$ colors (rather than $2$)  and $\lambda$-balance, a more flexible restriction on  the multiset of colors  in a vertex neighborhood.  We also provide techniques for constructing examples of graphs in one or more of these classes. The  balance number for each class is the minimum $\lambda$ for which a graph has a $k$-coloring that is $\lambda$-balanced and this is introduced in Section~\ref{sec-bal-number}.  Each of the three balance numbers can be arbitrarily large, but we prove inequalities relating the balance numbers  to one another and to  the maximum degree, and prove these inequalities are tight. In Sections~\ref{sec-pb-2-1}, \ref{caterpillar-sec}, and \ref{complete-multipartite-sec},
we restrict attention to the important case where $k=2$ and $\lambda \le 1$ and introduce a fourth class that generalizes NBC and CNBC graphs. Our Venn diagram (Figure \ref{fig-venn}) shows containments and separating examples for the four classes. We introduce the technique of red-blue removals to prove balance number results. We also focus on particular graph classes:  paths, cycles, wheels, trees, caterpillars, and complete multipartite graphs.

\section{Color $2$-switches}
\label{sec-two-switch}

For two graphs with  the same degree sequence,  a classic result in graph theory provides a way to transform one to the other using a sequence of $2$-switches (e.g., see \cite[p. 47]{We01}).  A \emph{$2$-switch} in graph $G$  in which  $ux,wy\in E(G)$, and $uy, wx\not\in E(G)$ is a replacement of the edges $ux$ and $wy$ with the edges $uy$ and $wx$. Figure \ref{house-fig} shows two non-isomorphic graphs that each have degree sequence $3,3,2,2,2.$   The $2$-switch 
in which edges $v_2v_5$ and $v_3v_4$ are replaced by $v_2v_4$ and $v_3v_5$ transforms $G$ to $H$.
In Theorem~\ref{two-switch-thm-1}, we prove an analog of this result for color degree matrices.

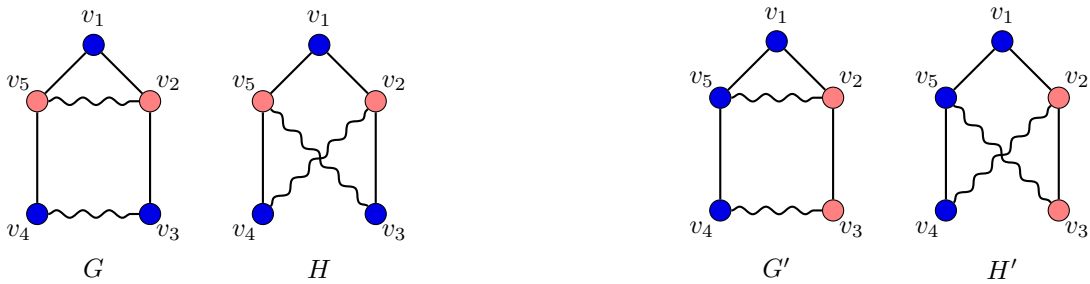
\begin{figure}[ht]
    \centering
    \begin{minipage}{0.45\textwidth}
    \centering
       \begin{tikzpicture}[scale=.75, every node/.style={circle, draw, minimum size=8pt, inner sep=0pt}]
    \node[fill=blue!90!black, label=above:$v_1$] (v3) at (1,3) {};
    \node[fill=red!50, label=above left:$v_5$] (v4) at (0,2) {};
    \node[fill=blue!90!black, label=below left:$v_4$] (v1) at (0,0) {};
    \node[fill=blue!90!black, label=below right:$v_3$] (v2) at (2,0) {};
    \node[fill=red!50, label=above right:$v_2$] (v5) at (2,2) {};
    \node[draw=none, fill=none, label=below:$G$] (label) at (1,-.5) {};
    \draw[thick] (v1)--(v4);
    \draw[thick] (v2)--(v5);
    \draw[thick] (v3)--(v4);
    \draw[thick] (v3)--(v5);
    \draw[thick, wiggly] (v1) -- (v2);
    \draw[thick, wiggly] (v4) -- (v5);

    \node[fill=blue!90!black, label=above:$v_1$] (v3) at (5,3) {};
    \node[fill=red!50, label=above left:$v_5$] (v4) at (4,2) {};
    \node[fill=blue!90!black, label=below left:$v_4$] (v1) at (4,0) {};
    \node[fill=blue!90!black, label=below right:$v_3$] (v2) at (6,0) {};
    \node[fill=red!50, label=above right:$v_2$] (v5) at (6,2) {};
    \node[draw=none, fill=none, label=below:$H$] (label) at (5,-.5) {};
    \draw[thick] (v1)--(v4);
    \draw[thick] (v2)--(v5);
    \draw[thick] (v3)--(v4);
    \draw[thick] (v3)--(v5);
    \draw[thick, wiggly] (v1) -- (v5);
    \draw[thick, wiggly] (v2) -- (v4);
\end{tikzpicture}
    \end{minipage}
    \hfill
    \begin{minipage}{0.45\textwidth}
    \centering
        \begin{tikzpicture}[scale=.75, every node/.style={circle, draw, minimum size=8pt, inner sep=0pt}]
    \node[fill=blue!90!black, label=above:$v_1$] (w2) at (1,3) {};
    \node[fill=blue!90!black, label=above left:$v_5$] (w1) at (0,2) {};
    \node[fill=blue!90!black, label=below left:$v_4$] (w3) at (0,0) {};
    \node[fill=red!50, label=below right:$v_3$] (w5) at (2,0) {};
    \node[fill=red!50, label=above right:$v_2$] (w4) at (2,2) {};
    \node[draw=none, fill=none, label=below:$G'$] (label) at (1,-.5) {};
    \draw[thick] (w1)--(w2);
    \draw[thick] (w1)--(w3);
    \draw[thick] (w2)--(w4);
    \draw[thick] (w4)--(w5);
    \draw[thick, wiggly] (w1) -- (w4);
    \draw[thick, wiggly] (w3) -- (w5);

    \node[fill=blue!90!black, label=above:$v_1$] (w2) at (5,3) {};
    \node[fill=blue!90!black, label=above left:$v_5$] (w1) at (4,2) {};
    \node[fill=blue!90!black, label=below left:$v_4$] (w3) at (4,0) {};
    \node[fill=red!50, label=below right:$v_3$] (w5) at (6,0) {};
    \node[fill=red!50, label=above right:$v_2$] (w4) at (6,2) {};
    \node[draw=none, fill=none, label=below:$H'$] (label) at (5,-.5) {};
    \draw[thick] (w1)--(w2);
    \draw[thick] (w1)--(w3);
    \draw[thick] (w2)--(w4);
    \draw[thick] (w4)--(w5);
    \draw[thick, wiggly] (w1)--(w5);
    \draw[thick, wiggly] (w3)--(w4);
\end{tikzpicture}
    \end{minipage}
    \caption{A 2-switch in which edges $v_2v_5$ and $v_3v_4$ are replaced by $v_2v_4$ and $v_3v_5$.  This is not a color $2$-switch for the first two graphs, but is a color $2$-switch  transforming $G'$ into $H'$. 
     }
    \label{house-fig}
\end{figure}

In order to maintain the multiset of colors in $N(v)$ for any vertex $v$, our  next definition takes into account the colors of the vertices in the $2$-switch.  Note that in Definition~\ref{2-switch} it is possible for all four of $u,w,x,y$ to be the same color.

\begin{defn} \label{2-switch} \rm Let $G$ be a graph in which each vertex has a color and let $u,w$ be vertices of the same color and let $x,y$ be vertices of the same color such that 
$ux,wy\in E(G)$, and $uy, wx\not\in E(G)$. A \emph{color $2$-switch} of $G$ with  this coloring   is a replacement of the edges $ux$ and $wy$ with the edges $uy$ and $wx$.
\end{defn}

 For a graph in which the vertices are assigned colors, the color of a vertex is the same after a $2$-switch, however, the colors of the vertices in its neighborhood may change.   For example, consider the colorings of graphs $G$ and $H$ in Figure~\ref{house-fig}, the 2-switch is \emph{not} a color 2-switch as $v_4$ and $v_5$ have different colors, and thus $v_2$ loses an adjacency to the red vertex $v_5$ and gains an adjacency to the blue vertex $v_4$.  The graphs $G$ and $H$ will have different color degree matrices no matter how the vertices are indexed because both red vertices of $H$ have three blue neighbors, while this is not true in $G$.
 For the $2$-colored graphs $G'$ and $H'$ in Figure~\ref{house-fig}, the 
 the same $2$-switch \emph{is} a color $2$-switch, and these will always preserve the number of red and blue neighbors of each vertex, as we note in the following remark.

\begin{remark} \label{rem:color-d-m}
    Color 2-switches do not affect the color of a vertex or the multiset of colors appearing among its neighbors. Thus, the color degree matrix of a graph remains the same after a color $2$-switch.
\end{remark}

\begin{figure}
    \centering
\[
D(G) = 
\left[
\begin{array}{ccc:c}
2 & 0 & 0 & {\text{\color{blue}B}}\\
1 & 2 & 0 & {\text{\color{red}R}}\\
1 & 1 & 0 & {\text{\color{blue}B}}\\
1 & 1 & 0 & {\text{\color{blue}B}}\\
1 & 2 & 0 & {\text{\color{red}R}}\\
\end{array}
\right]
\qquad
D(H) = 
\left[
\begin{array}{ccc:c}
2 & 0 & 0 & {\text{\color{blue}B}} \\
0 & 3 & 0 & {\text{\color{red}R}}\\
2 & 0 & 0 & {\text{\color{blue}B}}\\
2 & 0 & 0 & {\text{\color{blue}B}}\\
0 & 3 & 0 & {\text{\color{red}R}}\\
\end{array}
\right]
\qquad
D(G') = 
\left[
\begin{array}{ccc:c}
1 & 1 &0 & {\text{\color{blue}B}}\\
1 & 2 &0& {\text{\color{red}R}} \\
1 & 1 &0& {\text{\color{red}R}}\\
1 & 1 &0& {\text{\color{blue}B}}\\
1 & 2 &0& {\text{\color{blue}B}}\\
\end{array}
\right]=D(H')
\]

\caption{  
The color degree matrices for the graphs in Figure~\ref{house-fig} with color 1  (red, {\color{red}R}) and color 2 (blue, {\color{blue}B}).
}
    \label{fig:colordegreematrix}
\end{figure}

Recall that in defining the color degree matrix of a graph with a $k$-coloring, we  incorporate the  color of each vertex, using the color-identifier column, as well as the number of neighbors it has in each color class.   The following example shows why the color-identifier column is crucial.

\begin{exa}
\rm
Figure~\ref{fig-two-trees} shows the trees $T_1$ and $T_2$ with a $2$-coloring using $C_1$ (red) and $C_2$ (blue).  The ordered pair $(i,j)$ listed at each vertex indicates that the vertex has $i$ red neighbors and $j$ blue neighbors.  These two trees have the same multiset of ordered pairs, so there exist orderings of their vertices so that the first two columns of their color degree matrices are the same. However, $T_1$ has $3$ red and $6$ blue vertices, while $T_2$ has $4$ red and $5$ blue vertices, so the color-identifier columns  of their color degree matrices  will be different and there is no color $2$-switch that can transform $T_1$ to $T_2$.
\end{exa}

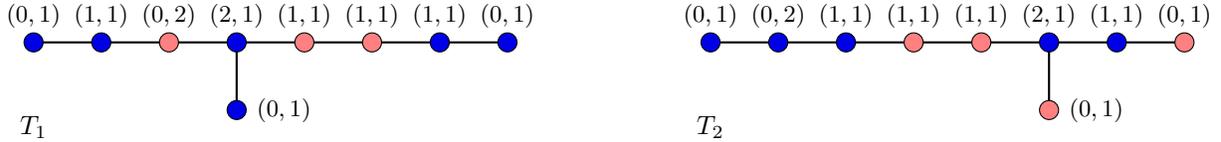
\begin{figure}[ht]
\centering
\begin{tikzpicture}[scale=.9]
\draw[thick] (0,1)--(7,1);
\draw[thick] (3,1)--(3,0);

\filldraw[blue!90!black, draw=black]
(0,1) circle [radius=4pt]
(1,1) circle [radius=4pt]
(3,1) circle [radius=4pt]
(6,1) circle [radius=4pt]
(7,1) circle [radius=4pt]
(3,0) circle [radius=4pt]
;
\filldraw[red!50, draw=black]
(2,1) circle [radius=4pt]
(4,1) circle [radius=4pt]
(5,1) circle [radius=4pt]
;

\draw[thick] (10,1)--(17,1);
\draw[thick] (15,1)--(15,0);
\filldraw[blue!90!black, draw=black]
(10,1) circle [radius=4pt]
(11,1) circle [radius=4pt]
(12,1) circle [radius=4pt]
(15,1) circle [radius=4pt]
(16,1) circle [radius=4pt]

;
\filldraw[red!50, draw=black]
(13,1) circle [radius=4pt]
(14,1) circle [radius=4pt]
(17,1) circle [radius=4pt]
(15,0) circle [radius=4pt]
;
\node(0) at (0,-.25) {$T_1$};
\node(0) at (10,-.25) {$T_2$};
\node(0) at (0,1.4) {\small $(0,1)$};
\node(0) at (1,1.4) {\small $(1,1)$};
\node(0) at (2,1.4) {\small $(0,2)$};
\node(0) at (3,1.4) {\small $(2,1)$};
\node(0) at (4,1.4) {\small $(1,1)$};
\node(0) at (5,1.4) {\small $(1,1)$};
\node(0) at (6,1.4) {\small $(1,1)$};
\node(0) at (7,1.4) {\small $(0,1)$};
\node(0) at (3.7,0) {\small $(0,1)$};

\node(0) at (10,1.4) {\small $(0,1)$};
\node(0) at (11,1.4) {\small $(0,2)$};
\node(0) at (12,1.4) {\small $(1,1)$};
\node(0) at (13,1.4) {\small $(1,1)$};
\node(0) at (14,1.4) {\small $(1,1)$};
\node(0) at (15,1.4) {\small $(2,1)$};
\node(0) at (16,1.4) {\small $(1,1)$};
\node(0) at (17,1.4) {\small $(0,1)$};
\node(0) at (15.7,0) {\small $(0,1)$};
\end{tikzpicture}
\caption{The trees $T_1$ and $T_2$ where  the   ordered pair $(i,j)$  next to a vertex $v$ indicates that $v$ has $i$ red neighbors and $j$ blue neighbors. }
\label{fig-two-trees}
\end{figure}

\begin{thm}
For $k$-colored graphs $G$ and $H$, the color degree matrices $D(G)$ and $D(H)$ are  equal if and only if there is a sequence of color 2-switches that takes $G$ to $H$.
\label{two-switch-thm-1}
\end{thm}

\begin{proof} The backwards direction follows from Remark~\ref{rem:color-d-m}.  
We will prove the forward direction by induction on $|V(G)|$.

Let $D(G)$ have $n$ rows and $k+1$ columns.  Let $G$ be a graph with $k$-coloring $C_1 \cup C_2\cup \cdots \cup C_k$ and let $D(G)$ be the associated color degree matrix. 
We will perform a sequence of color 2-switches on $G$.   
Select $v\in C_1$ with the largest $C_1$-degree in $G$. 
Order the vertices of $C_2$ by their $C_1$-degree in $G$, largest to smallest, as $w_1, w_2, \ldots, w_p$.  We begin by considering the neighbors of $v$ in $C_2$.
Let $q = C_2$-$\deg_G(v)$ and let $S_2 = \{w_1, w_2, \ldots, w_q\}$ be the $q$ vertices in $C_2$ with largest $C_1$-degree.  We will make color $2$-switches in $G$ so that in the resulting graph the set of neighbors of $v$ colored $C_2$ will be $S_2$.
If $N_G(v) \cap C_2 = S_2$, no color $2$-switches are necessary. 
Otherwise, there exist  $x,z\in C_2$ such that $x \in S_2$ with $x \not\in N_G(v) $ 
and $z \in N_G(v) $ 
with $z \not\in S_2$. By definition of $S_2$, we know $C_1$-$\deg_G(x) \ge C_1$-$\deg_G(z)$.  
Since $z\in N_G(v)$ and $x\not \in N_G(v)$, there exists a vertex $y\in C_1$ for which $y$ is adjacent to $x$ but not to $z$.

Perform a color $2$-switch by adding edges $vx$ and $zy$ and removing edges $vz$ and $xy$.  By Remark~\ref{rem:color-d-m}, the result is a graph that has the same color degree matrix as $G$, but for which we have increased the quantity $|(N_G(v) \cap C_2) \cap S_2|$. Repeat this process by induction until we get a graph with the same color degree matrix as $G$ and $N_G(v) \cap C_2 = S_2$.

Next, repeat this process for each of the  other color classes $C_3, C_4, \ldots, C_k$ so that for $j\ge 1$, the set $N_G(v) \cap C_j$  consists of  $C_j$-$\deg_G(v)$ vertices of $C_j$ with largest $C_1$-degree, and call this set $S_j$. 
When we make color $2$-switches for color $j$, they only involve  edges with one  endpoint in $C_1$   and the other in $C_j$, so these changes do not affect the previous color $2$-switches on  edges with one endpoint in $C_1$ and the other in  $C_i$  where $i<j$. Call the resulting graph $G_1$. 

Finally, we consider the neighbors of $v$ in $C_1$.
Order the vertices in $C_1$, other than $v$, 
as $v_1, v_2, \ldots, v_{\ell}$ by their $C_1$ degree, largest to smallest, so that $C_1$-$\deg_G(v_1) \ge C_1$-$\deg_G(v_2) \ge \cdots \ge C_1$-$\deg_G(v_{\ell})$.  Let  $m = |N_G(v) \cap C_1|$ and let $S_1 = \{v_1, v_2, 
\ldots, v_{m} \}$. Note that $|S_1| = C_1$-$\deg_G(v)$. As before, we can make a sequence of color $2$-switches, 
where all four vertices involved have color $C_1$, to transform $G_1$ to a graph $G^*$ that has the same color degree matrix as $G$ and for 
which $S_1 = N_G(v) \cap C_1$. Since these color $2$-switches involve adding and deleting edges that have both 
endpoints in $C_1$, the set of  neighbors of $v$ of colors other than $C_1$ remains the same. In $G^*$, we have 
$N_{G^\ast}(v) = S_1 \cup S_2 \cup \cdots \cup S_k$.

By hypothesis,   $D(H) = D(G)$, so there exists another sequence of color $2$-switches that transforms $H$ to $H^\ast$  where $N_{H^\ast}(v) = S_1 \cup S_2 \cup \cdots \cup S_k$.  Now, $G^{\ast} - v$ and $H^{\ast} - v$ have the same color degree matrix.  By our inductive hypothesis, there is a sequence of color $2$-switches that transforms $G^* -v$ to $H^* - v$.  These color $2$-switches do not involve $v$, hence the same sequence of color $2$-switches that transforms $G^*-v$ to $H^*-v$ transforms $G^*$ to $H^*$. Thus, we have a sequence of color 2-switches that transforms $G$ to $G^*$, and from $G^*$ to $H^*$, and the reverse of the sequence we constructed from $H$ to $H^*$ transforms $H^*$ to $H$. Hence, there is a sequence of color 2-switches that transforms $G$ to $H$.
\end{proof}

Theorem~\ref{two-switch-thm-1} is valuable in that
for two large $k$-colored graphs, it may be easier to determine whether    they have the same color degree matrix than it is to determine whether there is a sequence of color $2$-switches that transforms one to the other.  Our next two examples demonstrate this.

 \begin{exa} \rm
One can check that the $3$-colored graphs  $G$ and $G'$ in Figure~\ref{sample2switches}  have the same color degree matrix, therefore by Theorem~\ref{two-switch-thm-1}, there is a  sequence of  color $2$-switches that transforms one to the other.  The proof of Theorem~\ref{two-switch-thm-1} provides 
a method for constructing such a sequence of $2$-switches in general.  For the graphs in Figure~\ref{sample2switches} the following sequence achieves the transformation:

(i) replace $\{v_3v_{12}, v_6v_9\}$ with $\{v_3v_{9}, v_6v_{12}\}$,

(ii) replace $\{v_1v_{11}, v_8v_{10}\}$ with $\{v_1v_{10}, v_8v_{11}\}$, and 

(iii) replace $\{v_2v_{4}, v_5v_{7}\}$ with $\{v_2v_{5}, v_4v_{7}\}$.
\end{exa}

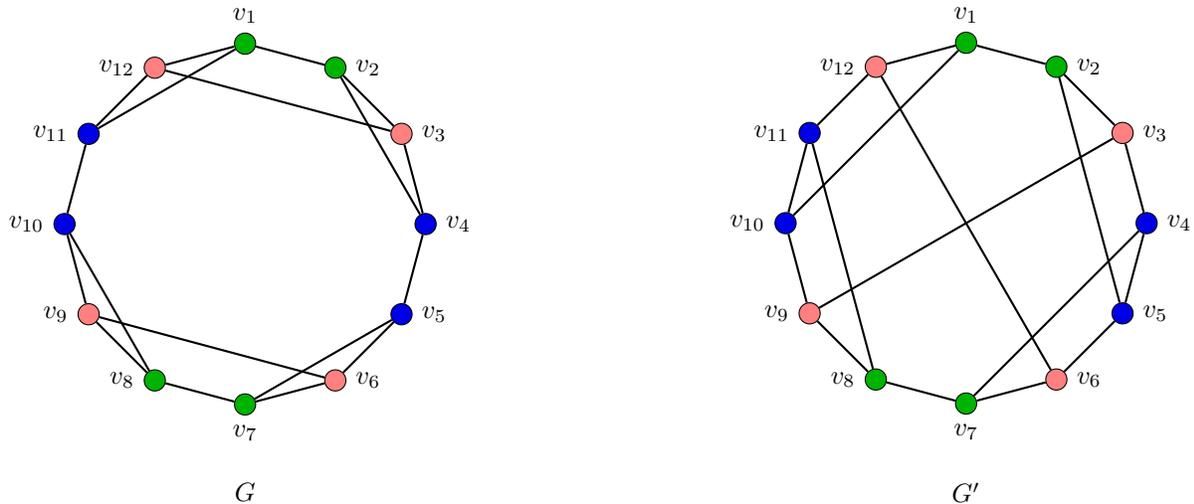
\begin{figure}
\begin{minipage}{.45\textwidth}
\begin{tikzpicture}[scale=.8]

\gvertex (1)  at ( 0, 3) [label=above: $v_1$]{};
\gvertex (2)  at ( 1.5, 2.6) [label=right: $v_2$] {};
\rvertex (3)  at ( 2.6, 1.5) [label=right: $v_3$] {};
\bvertex  (4)  at ( 3, 0) [label=right: $v_4$] {};
\bvertex (5) at ( 2.6,-1.5)[label=right: $v_5$] {};
\rvertex (6) at (1.5,-2.6) [label=right: $v_6$] {};
\gvertex (7) at ( 0,-3) [label=below: $v_7$]{};
\gvertex (8)  at (-1.5,-2.6) [label=left: $v_8$] {};
\rvertex (9)  at (-2.6,-1.5) [label=left: $v_9$] {};
\bvertex (10)  at (-3, 0) [label=left: $v_{10}$] {};
\bvertex (11)  at (-2.6, 1.5) [label=left: $v_{11}$] {};
\rvertex (12) at (-1.5, 2.6) [label=left: $v_{12}$] {};

\draw[-, thick] (1) to (2) to (3) to (4) to (5) to (6) to (7) to (8) to (9) to (10) to (11) to (12) to (1);

\draw[-, thick] (1) to (11);
\draw[-, thick] (2) to (4);
\draw[-, thick] (3) to (12);
\draw[-, thick] (5) to (7);
\draw[-, thick] (6) to (9);
\draw[-, thick] (8) to (10);

\node[draw=none, fill=none, label=below:$G$] (label) at (0,-4) {};
\end{tikzpicture}
\end{minipage}
\hspace{20mm}
\begin{minipage}{.45\textwidth}
\begin{tikzpicture}[scale=.8]

\gvertex (1)  at ( 0, 3) [label=above: $v_1$]{};
\gvertex (2)  at ( 1.5, 2.6) [label=right: $v_2$] {};
\rvertex (3)  at ( 2.6, 1.5) [label=right: $v_3$] {};
\bvertex  (4)  at ( 3, 0) [label=right: $v_4$] {};
\bvertex (5) at ( 2.6,-1.5)[label=right: $v_5$] {};
\rvertex (6) at (1.5,-2.6) [label=right: $v_6$] {};
\gvertex (7) at ( 0,-3) [label=below: $v_7$]{};
\gvertex (8)  at (-1.5,-2.6) [label=left: $v_8$] {};
\rvertex (9)  at (-2.6,-1.5) [label=left: $v_9$] {};
\bvertex (10)  at (-3, 0) [label=left: $v_{10}$] {};
\bvertex (11)  at (-2.6, 1.5) [label=left: $v_{11}$] {};
\rvertex (12) at (-1.5, 2.6) [label=left: $v_{12}$] {};

\draw[-, thick] (1) to (2) to (3) to (4) to (5) to (6) to (7) to (8) to (9) to (10) to (11) to (12) to (1);

\draw[-, thick] (1) to (10);
\draw[-, thick] (2) to (5);
\draw[-, thick] (3) to (9);
\draw[-, thick] (4) to (7);
\draw[-, thick] (6) to (12);
\draw[-, thick] (8) to (11);

\node[draw=none, fill=none, label=below:$G'$] (label) at (0,-4) {};
\end{tikzpicture}
\end{minipage}
\caption{Graphs $G$ and $G'$ and $3$-colorings of  them for which $D(G) = D(G')$.
}
\label{sample2switches}
\end{figure}

\begin{exa}
\rm

Figure~\ref{fig-OBS-no-color-2switch} shows $2$-colorings of the graphs  $ K_2 \Box C_5$ and $K_4+K_3\Box K_2$, which have many common features. Both graphs are regular of degree $3$,  the colorings shown each have 5 red and 5 blue vertices, and the open neighborhood of each vertex contains one vertex of one color and two vertices of the opposite color.  Their color degree matrices are shown in Figure~\ref{fig:matrix2} and  the first two columns of each of these matrices have five entries of $2,1$ and five entries of $1,2$.    
However, the color degree matrices are not equal because in  $D( K_2 \Box C_5)$ there are four $[2,1, {\text{\color{blue}B}]}$ rows and one $[2,1, {\text{\color{red}R} }]$ row while in     $D(K_4+K_3\Box K_2)$ there are two of the former and three of the latter.  Therefore, by Theorem~\ref{two-switch-thm-1}, there is no sequence of color $2$-switches that transforms our coloring of  $ K_2 \Box C_5$ to our coloring of  $K_4+K_3\Box K_2$.

\end{exa}

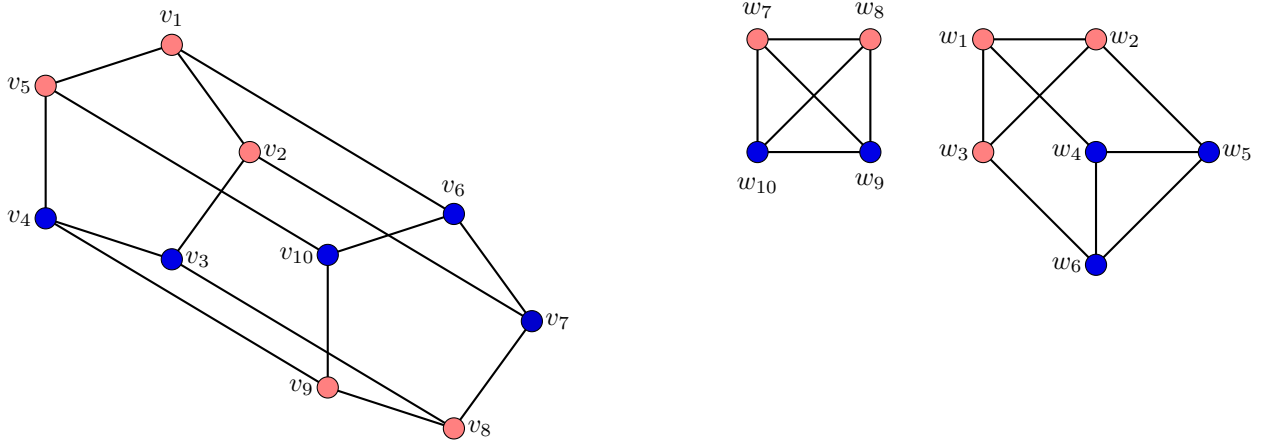
\begin{figure}[ht]
\begin{tikzpicture}[scale=.75, every node/.style={circle, draw, minimum size=8pt, inner sep=0pt}]

    \def\r{2}
    \foreach \i in {1,...,5} {
        \coordinate (v\i) at ({\r*cos(72*\i)}, {\r*sin(72*\i)});
    }
    \def\xshift{5}
    \def\yshift{-3}
    \foreach \i in {1,...,5} {
        \coordinate (w\i) at ({\xshift + \r*cos(72*\i)}, {\yshift + \r*sin(72*\i)});
    }
    \draw[thick] (v1)--(v2)--(v3)--(v4)--(v5)--(v1);
    \draw[thick] (w1)--(w2)--(w3)--(w4)--(w5)--(w1);
    \draw[thick] (v1)--(w1);
    \draw[thick] (v2)--(w2);
    \draw[thick] (v3)--(w3);
    \draw[thick] (v4)--(w4);
    \draw[thick] (v5)--(w5);

    \node[fill=red!50, label=above:$v_1$] at (v1) {};
    \node[fill=red!50, label=left:$v_5$] at (v2) {};
    \node[fill=blue!90!black, label=left:$v_4$] at (v3) {};
    \node[fill=blue!90!black, label=right:$v_3$] at (v4) {};
    \node[fill=red!50, label=right:$v_2$] at (v5) {};
   
    \node[fill=blue!90!black, label=above:$v_6$] at (w1) {};
    \node[fill=blue!90!black, label=left:$v_{10}$] at (w2) {};
    \node[fill=red!50, label=left:$v_9$] at (w3) {};
    \node[fill=red!50, label=right:$v_8$] at (w4) {};
    \node[fill=blue!80!black, label=right:$v_7$] at (w5) {};

\begin{scope}[xshift=11cm]
    \coordinate (a1) at (0,2);
    \coordinate (a2) at (2,2);
    \coordinate (a3) at (2,0);
    \coordinate (a4) at (0,0);

    \draw[thick] (a1)--(a2)--(a3)--(a4)--(a1);
    \draw[thick] (a1)--(a3);
    \draw[thick] (a2)--(a4);

   \node[fill=red!50, label=above:$w_7$] at (a1) {};
    \node[fill=red!50, label=above:$w_8$] at (a2) {};
    \node[fill=blue!90!black, label=below:$w_9$] at (a3) {};
    \node[fill=blue!90!black, label=below:$w_{10}$] at (a4) {};
\end{scope}

\begin{scope}[xshift=15cm]
    \coordinate (b1) at (0,2);
    \coordinate (b2) at (2,2);
    \coordinate (b3) at (0,0);
    \coordinate (b4) at (2,0);
    \coordinate (b5) at (2,-2);
    \coordinate (b6) at (4,0);

    \draw[thick] (b1)--(b2)--(b3)--(b1);
    \draw[thick] (b4)--(b5)--(b6)--(b4);
    \draw[thick] (b1)--(b4);
    \draw[thick] (b3)--(b5);
    \draw[thick] (b2)--(b6);
    
    \node[fill=red!50, label=left:$w_1$] at (b1) {};
    \node[fill=red!50, label=right:$w_2$] at (b2) {};
    \node[fill=red!50, label=left:$w_3$] at (b3) {};
    \node[fill=blue!90!black, label=left:$w_4$] at (b4) {};
    \node[fill=blue!90!black, label=left:$w_6$] at (b5) {};
    \node[fill=blue!90!black, label=right:$w_5$] at (b6) {};
\end{scope}

\end{tikzpicture}
\caption{The graphs $K_2 \Box C_5$ and $K_4+K_3\Box K_2$ with $2$-colorings and no possible transformation via color 2-switches. }
\label{fig-OBS-no-color-2switch}
\end{figure}

\begin{figure}[ht]
    \centering
\[
D(K_2\Box C_5) = 
\left[
\begin{array}{cc:c}
2 & 1 & {\text{\color{red}R}} \\
1 & 2 & {\text{\color{red}R}} \\
2 & 1 & {\text{\color{blue}B}} \\
2 & 1 & {\text{\color{blue}B}}\\
1 & 2 & {\text{\color{red}R}} \\
1 & 2 & {\text{\color{blue}B}}\\
2 & 1 & {\text{\color{blue}B}} \\
1 & 2 & {\text{\color{red}R}} \\
1 & 2 & {\text{\color{red}R}}\\
2 & 1 & {\text{\color{blue}B}}  \\
\end{array}
\right]
\qquad
D(K_4+K_3\Box K_2) = 
\left[
\begin{array}{cc:c}
2 & 1 & {\text{\color{red}R}}\\
2 & 1 & {\text{\color{red}R}}\\
2 & 1 & {\text{\color{red}R}}\\
1 & 2 & {\text{\color{blue}B}}\\
1 & 2 & {\text{\color{blue}B}}\\
1 & 2 & {\text{\color{blue}B}}\\
1 & 2 & {\text{\color{red}R}}\\
1 & 2 & {\text{\color{red}R}}\\
2 & 1 & {\text{\color{blue}B}}\\
2 & 1 & {\text{\color{blue}B}}\\
\end{array}
\right]
\]
\caption{Color degree matrices for Figure \ref{fig-OBS-no-color-2switch}.}
\label{fig:matrix2}
\end{figure}

\section {Three classes of $k$-colored and $\lambda$-balanced graphs}
\label{sec-three-classes}

In this section, we introduce three versions of $\lambda$-balance for $k$-colored graphs, discuss fundamental properties, and provide illustrative examples.

For a graph with $k$-coloring $C_1 \cup C_2 \cup \cdots C_k$, we defined the  $C_j$-degree of a vertex in Section~\ref{sec-intro}.  We now define the \emph{color $j$ closed  degree} of vertex $x$ to be $|N[x] \cap C_j|$ and denote it by $C_j$-$\deg[x]$.  In our next definitions, we relax the  notion of balanced, so that the number of vertices of each color in a neighborhood need not be exactly equal.
Let $G$ be a graph and fix a $k$-coloring $C_1 \cup C_2 \cup \cdots \cup C_k$ of $G$. For vertex $v$, we say the coloring is \emph{$\lambda$-balanced} at $N(v)$ if for all $i,j \in \{1,2, \ldots, k\}$, we have $|C_i$-$\deg(v) - C_j$-$\deg(v)| \le \lambda$. Likewise, it is \emph{$\lambda$-balanced} at $N[v]$ if $|C_i$-$\deg[v] - C_j$-$\deg[v]| \le \lambda$ for all $i,j \in \{1, 2, \ldots, k\}$. The $2$-colored graphs in Figure~\ref{fig-OBS-no-color-2switch} are $1$-balanced at $N(v)$ for each vertex $v$ and this can be seen easily from the color degree matrices in Figure~\ref{fig:matrix2}.  
The $3$-colored graphs in Figure~\ref{sample2switches} are $0$-balanced at $N(v)$ for each vertex $v$ since the neighborhood of each vertex consists of one vertex of each color.

 Our first graph class  is a generalization of NBC graphs and requires that the open neighborhoods of all vertices are $\lambda$-balanced. Indeed, the class of NBC graphs results when $k=2$ and $\lambda = 0$.

\begin{defn}
\label{open-bal-def}
\rm A $k$-coloring of graph $G$ that is $\lambda$-balanced at $N(v)$ for every $v \in V(G)$ is  called a  \emph {$(k,\lambda)$-balanced}  coloring and a graph with such a coloring is  a   \emph{$(k,\lambda)$-balanced}
graph.
 
\end{defn}

An analogous definition for closed neighborhoods generalizes CNBC-graphs and the class of CNBC graphs results when $k=2$ and $\lambda = 0$.

\begin{defn}
\rm 
A $k$-coloring of graph $G$ that is $\lambda$-balanced at $N[v]$ for every $v \in V(G)$ is  called a  \emph {$[k,\lambda]$-balanced}  coloring and a graph with such a coloring is  a   \emph{$[k,\lambda]$-balanced}
graph. 
\label{closed-bal-def}
\end{defn}

The house graph in Figure~\ref{house-fig} displays a $[2,1]$-balanced coloring of $G$ and a $(2,1)$-balanced coloring of $G'$.   Note that the $2$-coloring shown for $G$ is not $(2,1)$-balanced since it is not $1$-balanced at $N(v_1)$. And, the $2$-coloring shown for $G'$ is not $[2,1]$-balanced since it is not $1$-balanced at $N[v_5]$. The $3$-colorings shown for  the graphs in Figure~\ref{sample2switches}  are $(3,0)$-balanced and $[3,1]$-balanced  while the $2$-colorings shown for the graphs in Figure~\ref{fig-OBS-no-color-2switch} are $(2,1)$-balanced and $[2,2]$-balanced.

Our third class of locally balanced graphs is the most flexible, where at each vertex it suffices for $\lambda$-balance to occur at either the open or the closed neighborhood. 

 \begin{defn}\label{bal-def}
\rm 
 A graph $G$ is $([k,\lambda])$-\emph{balanced} or \emph{locally balanced} for parameters $k$ and $\lambda$, if there exists a $k$-coloring of $G$  that is either $\lambda$-balanced at $N(v)$ or $\lambda$-balanced at $N[v]$ for every $v \in V(G)$.  Such a coloring is a \emph{$([k,\lambda])$-balanced} coloring.

 \end{defn}

Any graph that is $(k,\lambda)$-balanced or  $[k,\lambda]$-balanced is also  $([k,\lambda])$-balanced, and we record this and similar consequences of the definitions  of these classes in the following lemma.

\begin{lem} 
The following hold for integers $k \ge 2$ and $\lambda \ge 0$.

\begin{enumerate}[label=(\roman*), font=\normalfont]
\item \label{lem:fundamentals-lemma01} Any graph with maximum degree $\Delta$ is $(k,\Delta)$-balanced and $[k,\Delta +1]$-balanced.

\item \label{lem:fundamentals-lemma02} Any graph that is $(k,\lambda)$-balanced or  $[k,\lambda]$-balanced is also  $([k,\lambda])$-balanced.  

\item \label{lem:fundamentals-lemma03} A graph is  $(k,\lambda)$-balanced if and only if each of its components is  $(k,\lambda)$-balanced.  Similar statements hold for  $[k,\lambda]$-balanced and $([k,\lambda])$-balanced. 

\item \label{lem:fundamentals-lemma04}If $G$ is $(k,\lambda)$-balanced then it is also $[k,\lambda + 1]$-balanced.  If $G$ is $[k,\lambda]$-balanced then it is also $(k,\lambda + 1)$-balanced. 
  If $G$ is $([k,\lambda])$-balanced then it is also $(k,\lambda + 1)$-balanced and $[k,\lambda +1]$-balanced. 

\item \label{lem:fundamentals-lemma05} If $c$ is a $(k,\lambda)$-balanced coloring of $G$, and $G'$ is obtained by a sequence of color $2$-switches, then $c$ is also a $(k,\lambda)$-balanced coloring of $G'$. Similar statements hold if $c$ is a $[k,\lambda]$-balanced coloring of $G$ or a $([k,\lambda])$-balanced coloring of $G$.  

\end{enumerate}

    \label{fundamentals-lemma}
\end{lem}

\begin{proof}
    The  first four results  are direct consequences of the Definitions~\ref{open-bal-def}, \ref{closed-bal-def}, and \ref{bal-def}  and the definitions of open and closed neighborhoods.
The last result also uses Remark~\ref{rem:color-d-m}.
\end{proof}

In contrast to Lemma~\ref{fundamentals-lemma}\ref{lem:fundamentals-lemma02}, we show in 
 Proposition~\ref{wheel-prop} that the family of wheel graphs of the form $W_{4n+6}$ for $n \ge 1$ are $([2,0])$-balanced but neither $(2,0)$-balanced nor $[2,0]$-balanced. 

Lemma~\ref{fundamentals-lemma}\ref{lem:fundamentals-lemma05} 
 provides a method for converting a known $(k,\lambda)$-balanced coloring of a  graph to a $(k,\lambda)$-balanced coloring of related graphs.  For example, Figure~\ref{sample2switches}  shows a $(3,0)$-balanced coloring of graph $G$ which is also a a $(3,0)$-balanced coloring of the graph $G'$, obtained by a sequence color $2$-switches.
Lemma~\ref{fundamentals-lemma}\ref{lem:fundamentals-lemma05} can also be used to show that a sequence of color $2$-switches does not exist, as in the following example.

\begin{exa} \rm 
The generalized Petersen graph $GP(n,d)$ is defined in \cite{W69} for $n,d$ positive integers and $1 \le d \le \frac{n-1}{2}$.  It consists of $2n$ vertices, each of degree $3$, where $G(5,2)$ is the usual Petersen graph.  In \cite{Coetal}, it is shown that  for $d_1$ odd and $d_2$ even, $G(2m,d_1)$ is a CNBC graph  (i.e., $[2,0]$-balanced) while $G(2m,d_2)$  is not. Thus, for any $[2,0]$-coloring of $G(2m,d_1)$, there is no sequence of color 2-switches that transforms it to a $2$-coloring of $G(2m,d_2)$.
    \label{GP-example}
\end{exa}

Parts \ref{lem:fundamentals-lemma03} and \ref{lem:fundamentals-lemma05}  of Lemma~\ref{fundamentals-lemma}  also 
suggest a method of creating examples that satisfy certain neighborhood coloring properties but not others.  For example, the cycle $C_4$ is $(2,0)$-balanced, and $[2,1]$-balanced, but not $[2,0]$-balanced, while 
the complete graph  $K_4$ is $[2,0]$-balanced,  and $(2,1)$-balanced, but not $(2,0)$-balanced.
By Lemma~\ref{fundamentals-lemma}\ref{lem:fundamentals-lemma03},  the disjoint union $C_4 + K_4$ is neither $(2,0)$-balanced nor $[2,0]$-balanced, but it is $(2,1)$-balanced and $[2,1]$-balanced.
If we seek a connected example of a graph with these properties, we can fix a $2$-coloring of $C_4 + K_4$  that is both $(2,1)$-balanced  and $[2,1]$-balanced, and perform a color $2$-switch where two vertices come from $C_4$ and the  other two from $K_4$. The next result generalizes this example.

\begin{prop} \label{prop:venn} 
Let $S$ be a set of any balanced neighborhood coloring properties, such as $(k,\lambda)$-balanced or $[k,\lambda]$-balanced  for some $k$ and $\lambda$. 
Let $T_1, T_2 \subset S$  and suppose $G_i$ is a  connected graph that satisfies  all of the properties in $T_i$ and none of the properties in $S-T_i$ for $i = 1,2$.  Then the disjoint union $G_1 + G_2$ satisfies all the properties in $T_1 \cap T_2$ and none of the properties in $S-(T_1\cap T_2)$.  Moreover, a connected example can be constructed by fixing $k$-colorings of $G_i$ satisfying  all the properties in $T_i$ for $i = 1,2$, and performing a color $2$-switch that includes two vertices from $G_1$ and two vertices from $G_2$.

\end{prop}

\begin{proof}
The proof follows from  Lemma~\ref{fundamentals-lemma}\ref{lem:fundamentals-lemma03} and Lemma~\ref{fundamentals-lemma}\ref{lem:fundamentals-lemma05}.
\end{proof}

\section{Balance number}
\label{sec-bal-number}

As discussed in the introduction, it is desirable to construct $k$-colorings of graphs so that the multiset of colors in any vertex neighborhood is balanced or close to balanced. For a graph $G$ and a fixed $k$, we seek to \emph{minimize} the value of $\lambda$ for which there is a $k$-coloring of $G$ that is $\lambda$-balanced.  This motivates  our next definition.

\begin{defn}
\label{bal-no-def}
{\rm  For a fixed positive integer $k$, the least $\lambda$ for which graph $G$ is $(k, \lambda)$-balanced is called the \emph{open $k$-balance number} of $G$ and denoted by $\beta_k(G)$. 
Similarly, the  least $\lambda$ for which graph $G$ is $[k, \lambda]$-balanced is called the \emph{closed $k$-balance number} of $G$ and denoted by $\beta_k[G]$, and the least $\lambda$ for which graph $G$ is $([k, \lambda])$-balanced is called the \emph{local $k$-balance number} of $G$ and denoted by $\beta_k([G])$.
}

\end{defn}

The next lemma translates  some of the results from Lemma~\ref{fundamentals-lemma}
using the terminology in Definition~\ref{bal-no-def}.
Note that the second inequality in Lemma~\ref{balance-lemma}\ref{bal-lem-pt1} is stronger than  its counterpart in Lemma~\ref{fundamentals-lemma}.  In Proposition~\ref{tight-prop}
we  show that each of the inequalities  in Lemma~\ref{balance-lemma}
is tight. 

\begin{lem} Let $G$ be a graph and $k\geq 2$. 
\begin{enumerate}[label=(\roman*), font=\normalfont]
\item \label{bal-lem-pt1}
If $G$ has maximum degree $\Delta$, then $\beta_k(G) \le \Delta$ and $\beta_k[G] \le \Delta $. 

\item \label{bal-lem-pt2}  \  $0 \le  \beta_k[G] - \beta_k([G]) \le 1$ and 
$0 \le  \beta_k(G) - \beta_k([G]) \le 1$

\item  \label{bal-lem-pt3} 
If  $G$ consists of the components $G_1, G_2, \ldots, G_t$, then $\beta_k(G) = \max_j\{\beta_k(G_j)\}$.    Similar statements hold for $\beta_k[G]$ and $\beta_k([G]).$ 

\item  \label{bal-lem-pt4} 
$-1 \le \beta_k[G] - \beta_k(G) \leq 1$

\item \label{bal-lem-pt5}
If $\beta_k(G) = \lambda$ and $c$ is a $(k,\lambda)$-balanced coloring of $G$, and $G'$ is obtained by a sequence of color $2$-switches, then 
 $\beta_k(G')\le \lambda$. Similar statements hold for $\beta_k[G]$ and $\beta_k([G])$.
\end{enumerate}

\label{balance-lemma}
    \end{lem}
    
 \begin{proof}
Parts \ref{bal-lem-pt2} - \ref{bal-lem-pt5} of Lemma~\ref{balance-lemma} follow from Lemma~\ref{fundamentals-lemma} together with  Definition~\ref{bal-no-def}.
 It remains to prove part \ref{bal-lem-pt1}.

 The inequalities $\beta_k(G) \le \Delta$ and $\beta_k[G] \le \Delta + 1  $ follow from Lemma~\ref{balance-lemma}\ref{bal-lem-pt1}   so we need only show that  $\beta_k[G] < \Delta + 1  $.  For a $k$-coloring of a graph with maximum degree $\Delta$, call a vertex $v$ \emph{extreme} if $v$ and all of its neighbors have the same color. Suppose for a contradiction that $G$ is a graph and $k$ is a positive integer  for which   $\beta_k[G] = \Delta + 1  $.  Thus  no $k$-coloring of $G$ is $[k,\Delta]$-balanced and therefore every $k$-coloring of $G$ will have an extreme vertex.   Fix a $k$-coloring of $G$ that has the minimum  possible number of extreme vertices and let $v$ be an extreme vertex.  Recolor $v$ so that it has a different color.  Now $v$ no longer has the same color as all of its neighbors, and recoloring $v$ does not create any new extreme vertices among its neighbors because they now have a different color from $v$.    Thus the new $k$-coloring has at least one fewer extreme vertex, a contradiction.     
\end{proof}

 The next lemma applies in the particular case where $k=2$.  

\begin{lem}
If every vertex of  graph $G$ has even degree, then $\beta_2(G)$ is even and $\beta_2[G]$ is odd.  If every  vertex of $G$ has odd degree, then $\beta_2[G]$ is even and  $\beta_2(G)$ is odd.
\label{even-deg-lem}
\end{lem}


\begin{proof}
    Let $G$ be a graph for which every vertex has even degree and fix any $2$-coloring of $G$ using colors red and blue.  Then, for any $v \in V(G)$,  the number of red  vertices in $N(v)$ has the same parity as the number of blue vertices in $N(v)$, and therefore the difference between these quantities is even. Consequently, $\beta_2(G)$ is even. Similarly, the number of red  vertices in $N[v]$ has the opposite parity from the number of blue vertices in $N[v]$, and therefore the difference between these quantities is odd and $\beta_2[G]$ is odd.  An analogous argument proves the second statement.
\end{proof}

\begin{prop}
The inequalities in Lemma~\ref{balance-lemma} are tight. 
    \label{tight-prop}
\end{prop}

\begin{proof}

First we  show that the  two inequalities in Lemma~\ref{balance-lemma}\ref{bal-lem-pt1} are tight.  Let $G$ be the cycle $C_n$ where $n$ is not a multiple of $4$.  In \cite{FM24} it is shown that for these values of $n$ the cycle $C_n$ is not an NBC graph, thus $\beta_2(G) \ge 1$.  Furthermore, by Lemma~\ref{even-deg-lem}, since every vertex of $G$ has even degree, we know $\beta_2(G) \ge 2$.  However, $\beta_2(G) \le \Delta(G) =2$  by Lemma~\ref{balance-lemma}\ref{bal-lem-pt1}.  Thus  $\beta_2(C_n) = \Delta(C_n)$ when $n$ is   not a multiple of $4$ and the inequality $\beta_k(G) \le \Delta$  is tight.

\smallskip

Now let $G$ be the cycle $C_n$ where $n$ is not a multiple of three.   We will show $\beta_3[C_n] = 2$.
For a contradiction, suppose there is a $[3,1]$-balanced coloring of $C_n$  whose vertices are labeled $v_1, v_2,  \ldots, v_n$ consecutively around the cycle. Each closed neighborhood must consist of one vertex of each color, so  the  colors must appear in sequence and without loss of generality we may assume vertex $v_i$ is red when $i \equiv 0 \pmod 3$,   blue when $i \equiv 1 \pmod 3$, and green $i \equiv 2 \pmod 3$.  Since $n$ is not a multiple of $3$, the coloring is not $1$-balanced at $N[v_n]$, a contradiction.  Thus $\beta_3[C_n] = 2 = \Delta$ when $n$ is not a multiple of $3$  and the inequality  $\beta_k[G] \le \Delta$ is tight.   

\smallskip
 
Next we focus on the six inequalities in Lemma~\ref{balance-lemma}\ref{bal-lem-pt2} and Lemma~\ref{balance-lemma}\ref{bal-lem-pt4}. Let $G$ be the complete graph $K_n$ where $n= tk$ for some positive integer $t$. Notice that $N[v] = V(G)$ for each $v \in V(G)$. The $k$-coloring in which there are $t$ vertices of each color is $0$-balanced at each closed neighborhood, so $\beta_2[G] = \beta_2([G]) = 0$.  
However, $|N(v)| = kt-1$ for each $v \in V(G)$,  and this quantity is not a multiple of $k$, so no $k$-coloring of $G$ is $0$-balanced  at $N(v)$ and hence 
$\beta_2(G) \ge 1$.  The  reverse inequality $\beta_2(G) \le 1$ follows from  Lemma~\ref{balance-lemma}\ref{bal-lem-pt4}, thus $\beta_2(G) = 1$. Hence any  graph $G$ of the form $K_{tk}$ demonstrates that the following inequalities in Lemma~\ref{balance-lemma}\ref{bal-lem-pt2} and Lemma~\ref{balance-lemma}\ref{bal-lem-pt4} are tight: 
$ 0 \le \beta_k[G] - \beta_k([G]) $, \
$\beta_k(G) -  \beta_k([G]) \le  1$, 
  and 
$ -1 \le \beta_k[G] - \beta_k(G)$.  

\smallskip

We use a family of complete bipartite graphs to show the remaining inequalities in Lemma~\ref{balance-lemma}\ref{bal-lem-pt2} and Lemma~\ref{balance-lemma}\ref{bal-lem-pt4} are tight.
Let $G = K_{tk,k+1}$  with bipartition $X \cup Y$ where $|X| = tk$  and $|Y| = k+1$ and $t$ is a positive integer.  Since $|N(x)|$ is not a multiple of $k$ for any $x \in X$,  there is no $(k,0)$-balanced coloring of $G$, so $\beta_k(G) \ge 1$.
 Consider the $k$-coloring in which there are $t$ vertices of each color in $X$ and  $Y$ contains two vertices of color 1 and one vertex of each of the other colors.  This $k$-coloring is $0$-balanced at  $N(y)$ for each $y \in Y$ and $1$-balanced at $x$ for each $x \in N(x)$.  Hence $\beta_k(G)  = 1$ and $\beta_k([G]) \le \beta_k(G)= 1.$    We  next show $\beta_k([G])  = 1$.  For a contradiction, assume $\beta_k([G])  = 0$ and fix a $([k,0])$-balanced coloring of $G$.  As before, for any $x \in X$, the coloring is not $0$-balanced at $N(x)$ because $|N(x)| = k+1$, so it must be $0$-balanced at $N[x]$. Thus for any two vertices $x,x' \in X$, observe that both $Y \cup \{x\}$ and $Y \cup \{x'\}$ contain an equal number of vertices of each color.  This implies that $x$ and $x'$ have the same color and so all vertices of $X$ have the same color.  This is a contradiction since  $k \ge 2$ and therefore the coloring will not be $0$-balanced at $N(y)$ or at $N[y]$ for any $y \in Y$.
 Thus, $\beta_k([G])  = 1$. 
\smallskip

 Next, we show $\beta_k[G]  = 2$.  We know $\beta_k[G]  \le 1+ \beta_k(G)= 2$ so it remains to show $\beta_k[G] > 1$.  For a contradiction, suppose there is a $[k,1]$- balanced coloring of $G$.  Since $|Y| = k+1$, we know there exists a color (say red) that appears on at least two vertices in $Y$ and another color (say blue) that appears on at most one vertex of $Y$.  For any red vertex $x \in X$, the closed neighborhood $N[x]$ contains at least three red and at most one blue vertex, a contradiction.  Consequently, there are no red vertices in $X$ and since $|X| \ge k$, there must be at least two vertices of the same color in $X$.   Now for any non-red vertex $y \in Y$, the closed neighborhood $N[y]$ contains no red vertices and at least two vertices of the same color, a contradiction.   Thus $\beta_k[G]  = 2$.
As a result any  graph $G$ of the form $K_{tk,k+1}$ demonstrates that the following inequalities in Lemma~\ref{balance-lemma} \ref{bal-lem-pt2} and Lemma~\ref{balance-lemma} \ref{bal-lem-pt4} are tight: 
$\beta_k[G] -  \beta_k([G]) \le  1$, 
\ $  0 \le  \beta_k(G) - \beta_k([G])   $, and 
$\beta_k[G] - \beta_k(G)  \le 1$.  
\end{proof}

Observe that the conclusion of Lemma~\ref{balance-lemma}\ref{bal-lem-pt5}
 is the  inequality  $\beta_k(G')\le \lambda$ rather than an equality.  The next example shows that the inequality can be strict.

\begin{exa} 
\rm 

In Figure~\ref{fig-balance}, the graph $G$ has a 2-coloring demonstrating $\beta_2(G) \le 2$. In fact, 
$\beta_2(G) = 2$, as we now show.  By Lemma~\ref{even-deg-lem} we know $\beta_2(G)$ is even, so it suffices to show $\beta_2(G) > 0$.  Suppose for a contradiction that there exists a $(2,0)$-balanced coloring of $G$ and without loss of generality we may assume $v_1$ is blue.  Since $|N(v_1)|=2$, one neighbor of $v_1$ must be red and the other blue, so  again without loss of generality,  we may assume $v_8$ is red and $v_2$ is blue.  Now $N(v_8)$ must contain two red and two blue vertices, so $v_6$ and $v_7$ are red, a contradiction since the coloring is not $0$-balanced at $N(v_7)$.  Thus, $\beta_2(G) = 2$.

A  color $2$-switch transforms  the $2$-colored graph $G$ into $G'$ by replacing the edges $v_2v_8$ and $v_4v_6$ with $v_2v_6$ and $v_4v_8$ as illustrated in Figure~\ref{fig-balance}. Thus, $\beta_2(G') \le 2$ by Lemma~\ref{balance-lemma}\ref{bal-lem-pt5}.    However, $\beta_2(G')=0$ and   a $(2,0)$-balanced coloring of $G'$ is shown in Figure~\ref{fig-balance}.
\label{ex-balance}
\end{exa}

While many of the graphs considered in this paper have a small balance number, there exist graphs, specifically bipartite graphs, for which $\beta_k(G)$, $\beta_k[G]$, and $\beta_k([G])$, are arbitrarily large.  The construction in the following theorem was communicated to us by Craig Larson \cite{Larson}.

\begin{thm}
For any integer $k \ge 2$, there exist bipartite graphs $G$ for which  $\beta_k(G)$, $\beta_k[G]$, and $\beta_k([G])$, are arbitrarily large. 
    \label{arb-large-thm}
\end{thm}

\begin{proof}
For integers $k \ge 2$ and $\lambda \ge 0$ we construct a bipartite graph $G$ for which $\beta_k(G) \ge \lambda$.  Let $V(G) = X \cup Y$ where $X = \{1,2,3, \ldots, (\lambda - 1)k+1\}$ and $Y$ is the set of all $\lambda$-element subsets of $X$.   An edge exists in $G$ between $x \in X$ and $S \in Y$ if and only if $x \in S$.  We fix any $k$-coloring of $G$ and show that it is not $(\lambda -1)$-balanced.  Since $|X| = (\lambda -1)k + 1$, there must be at least $\lambda$ vertices in $X$ that have the same color.  Without loss of generality we may assume that vertices $1,2,3, \ldots, \lambda$ in $X$ each have color 1.  The open neighborhood of vertex $\{1,2,3, \ldots, \lambda\} \in Y$ consists of $\lambda$ vertices of color $1$ and no vertices of color $2$, so the coloring is not 
$(\lambda -1)$-balanced.  Thus $\beta_k(G) \ge \lambda$.

By Lemma~\ref{balance-lemma}\ref{bal-lem-pt4},
  we know $\beta_k[G] \ge \beta_k(G) -1 \ge  \lambda-1$ and   by  Lemma~\ref{balance-lemma}\ref{bal-lem-pt2} we know $\beta_k([G])   \ge \beta_k(G) -1 \ge \lambda-1$, completing the proof.
\end{proof}

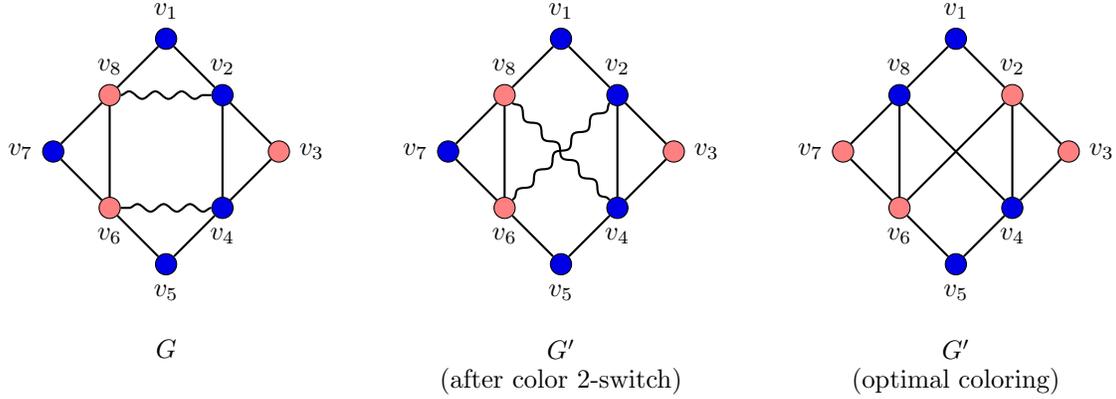
\begin{figure}\centering
\begin{tikzpicture}[scale=.75]
\bvertex (1) at (0,0) [label= left: $v_7$] {};
\rvertex (2) at (1,1) [label= above: $v_8$] {};
\bvertex (3) at (2,2) [label= above: $v_1$] {};
\bvertex (4) at (3,1) [label= above: $v_2$] {};
\rvertex (5) at (4,0) [label=right: $v_3$] {};
\bvertex (6) at (3,-1) [label= below: $v_4$] {};
\bvertex (7) at (2,-2) [label= below: $v_5$] {};
\rvertex (8) at (1,-1) [label = below: $v_6$] {};

\draw[-, thick] (1) to (2) to (3) to (4) to (5) to (6) to (7) to (8) to (1);
\draw[-, thick] (4) to (6);
\draw[-, thick] (8) to (2);
\draw[thick, wiggly] (2)--(4);
\draw[thick, wiggly] (6)--(8);

\node[draw=none, fill=none, label=below:$G$] (label) at (2,-3) {};

\bvertex (1) at (7,0) [label= left: $v_7$] {};
\rvertex (2) at (8,1) [label= above: $v_8$] {};
\bvertex (3) at (9,2) [label= above: $v_1$] {};
\bvertex (4) at (10,1) [label= above: $v_2$] {};
\rvertex (5) at (11,0) [label=right: $v_3$] {};
\bvertex (6) at (10,-1) [label= below: $v_4$] {};
\bvertex (7) at (9,-2) [label= below: $v_5$] {};
\rvertex (8) at (8,-1) [label = below: $v_6$] {};

\draw[-, thick] (1) to (2) to (3) to (4) to (5) to (6) to (7) to (8) to (1);
\draw[thick, wiggly] (2) to (6);
\draw[-, thick] (2) to (8);
\draw[-, thick] (4) to (6);
\draw[thick, wiggly] (4) to (8);

\node[draw=none, fill=none, label=below:$G'$] (label) at (9,-3) {};
\node[draw=none, fill=none, label=below:$\text{(after color 2-switch)}$] (label) at (9,-3.5) {};

\rvertex (1) at (14,0) [label= left: $v_7$] {};
\bvertex (2) at (15,1) [label= above: $v_8$] {};
\bvertex (3) at (16,2) [label= above: $v_1$] {};
\rvertex (4) at (17,1) [label= above: $v_2$] {};
\rvertex (5) at (18,0) [label=right: $v_3$] {};
\bvertex (6) at (17,-1) [label= below: $v_4$] {};
\bvertex (7) at (16,-2) [label= below: $v_5$] {};
\rvertex (8) at (15,-1) [label = below: $v_6$] {};

\draw[-, thick] (1) to (2) to (3) to (4) to (5) to (6) to (7) to (8) to (1);
\draw[-, thick] (2) to (6);
\draw[-, thick] (2) to (8);
\draw[-, thick] (4) to (6);
\draw[-, thick] (4) to (8);

\node[draw=none, fill=none, label=below:$G'$] (label) at (16,-3) {};
\node[draw=none, fill=none, label=below:$\text{(optimal coloring)}$] (label) at (16,-3.5) {};

\end{tikzpicture}
\caption{For the graph $G$, $\beta_2(G)=2$. A color 2-switch in which edges $v_2v_8$ and $v_4v_6$ are replaced by $v_2v_6$ and $v_4v_8$ transforms $G$ into $G'$. But, $\beta_2(G')=0$.
}
\label{fig-balance} 
\end{figure}

We end this section by showing that the quantities $\beta_3(G)$,  $\beta_3[G]$, and $\beta_3([G])$  are all equal when $G$ is   the Petersen graph.

\begin{figure}[ht]\centering
\begin{tikzpicture}[scale=.75, every node/.style={circle, draw, minimum size=8pt, inner sep=0pt}]

  \node (u0) at (0,2.5)              [fill=green!70!black, label=above:$v_{1}$] {};
  \node (u1) at (2.38,0.77)          [fill=blue!90!black, label=right:$v_2$] {};
  \node (u2) at (1.47,-2.02)         [fill=red!50, label=below right:$v_3$] {};
  \node (u3) at (-1.47,-2.02)        [fill=green!70!black, label=below left:$v_4$] {};
  \node (u4) at (-2.38,0.77)         [fill=red!50, label=left:$v_5$] {};

  \node (w0) at (0,1.2)              [fill=blue!90!black, label=above right:$v_6$] {};
  \node (w1) at (1.14,0.37)          [fill=green!70!black, label= below right:$v_7$] {};
  \node (w2) at (0.70,-0.97)         [fill=green!70!black, label= right:$v_8$] {};
  \node (w3) at (-0.70,-0.97)        [fill=blue!90!black, label=left:$v_9$] {};
  \node (w4) at (-1.14,0.37)         [fill=red!50, label=below left:$v_{10}$] {};

  \draw (u0)--(u1)--(u2)--(u3)--(u4)--(u0);

  \draw (w0)--(w2)--(w4)--(w1)--(w3)--(w0);

  \draw (u0)--(w0);
  \draw (u1)--(w1);
  \draw (u2)--(w2);
  \draw (u3)--(w3);
  \draw (u4)--(w4);

\end{tikzpicture}
\caption{A 3-coloring of the Petersen graph that is  (3,2)-balanced, [3,2]-balanced and  therefore ([3,2])-balanced. }
\label{fig-Pet-3color}
\end{figure}
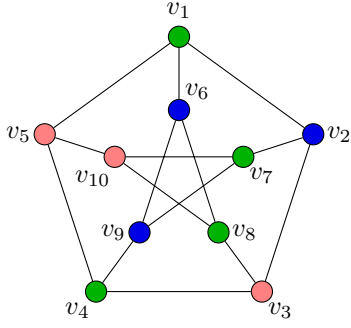

\begin{prop}    
 If $G$ is the Petersen graph, then  $\beta_3(G) = \beta_3[G] = \beta_3([G])  = 2$.
 \label{Petersen-prop}
 \end{prop}

 \begin{proof}
Figure~\ref{fig-Pet-3color} shows a $3$-coloring of the Petersen graph $G$ that is simultaneously $(3,2)$-balanced,  $[3,2]$-balanced, and   $([3,2])$-balanced.  Therefore, $\beta_3(G) \le 2$, \  $\beta_3[G] \le 2$, and $\beta_3([G]) \le 2$.  It remains to show the reverse inequalities.   First we show  $\beta_3(G) > 1$.  For a contradiction, suppose there is a $3$-coloring of $G$ (using colors red, blue, and green)  that is $1$-balanced at $N(v)$ for each $v \in V(G)$.   Thus the open neighborhood of each vertex must contain one vertex of each color. Applying this to $v_{1}$, we may assume, without loss of generality,  that $v_5$ is red, $v_6$ is blue, and $v_2$ is green.  If $v_{10}$ were green then $v_7$ would have two green neighbors, a contradiction.  Similarly, if $v_4$ were green then $v_3$ would have two green neighbors, a contradiction. Thus the green neighbor of $v_5$ must be $v_{1}$.  A symmetric argument applied to the neighbors of $v_2$  and using the color red in place of green, implies that $v_{1}$ must be red, a contradiction.

Next we show that $\beta_3[G] > 1$.  Suppose for a contradiction that there is a $3$-coloring of $G$ (using colors red, blue, and green)  that is $1$-balanced at $N[v]$ for each $v \in V(G)$.   Since we already showed that $\beta_3(G) > 1$, we know this coloring must have a vertex for which two of its neighbors have the same color.  Without loss of generality we may assume $v_{1}$ is blue, $v_5$ and $v_6$ are red, and $v_2$ is green.   Since $v_5$ must have a green neighbor, we know $v_4$ or $v_{10}$ must be green.  If $v_{10}$ is green then $v_7$ is red or blue (because it has two green neighbors), $v_9$ is red or blue (because $v_7$ has two green neighbors) and then $v_4$ is green (because the other neighbors  of $v_9$ are red or blue).  So in fact $v_4$ must be green regardless of whether $v_{10}$ is green.  Now $v_3$ must be red or blue (because it has two green neighbors), and $v_8$ must be red or blue (because $v_3$ already has two green neighbors), so $v_9$ is green (because $v_6$ must have a green neighbor).  Now $v_7$ is red or blue (because it has two green neighbors) and $v_{10}$ green (because its neighbors are all red or blue).  This is a contradiction because $v_7$ has three green neighbors.  

\smallskip
Finally, we show $\beta_3([G])=2$.   Since $G$ is $3$-regular,  for any vertex $v$, any $3$-coloring that is $1$-balanced at $N(v)$ is in fact $0$-balanced at $N(v)$, and therefore  is $1$-balanced at $N[v]$.  Thus, $\beta_3([G]) = \beta_3[G] = 2.$
\end{proof}

\section{Four classes of $2$-colored and $1$-balanced graphs} 
\label{sec-pb-2-1}

We return to the original instance of the problem in which there are two colors, red and blue, and this models situations where there are two alternatives.  We have seen in Theorem~\ref{arb-large-thm} that any of the three types of balance number can be arbitrarily large and for applications the difference between a graph having balance number $\lambda$ versus balance number $\lambda + 1$ may not be that significant once $\lambda \ge 2$.   Thus   in the remainder of the paper we focus on graphs that are not only $2$-colored but also $1$-balanced.  Recall that $\NBC$ is the class of $(2,0)$-balanced graphs and $\CNBC$ is the class of $[2,0]$-balanced graphs.  If a graph $G$ has both even and odd degree vertices, then   $ G \not\in \NBC$ and  $G \not\in\CNBC$, but it is possible for $G$ to be $(2,1)$-balanced and/or $[2,1]$-balanced.  


Just as NBC and CNBC are useful abbreviations for graphs with $\beta_2(G) =0$ and $\beta_2[G] =0$  respectively, we refer to graphs $G$ with 
$\beta_2(G) \le 1 $ as \emph{open semi-balanced }(OSB), those with $\beta_2[G] \le 1 $ as \emph{closed semi-balanced} (CSB), and  those with $\beta_2([G]) \le 1 $ as  \emph{semi-balanced  at each vertex} (SBV).  We  denote the set of all OSB graphs as $\OSB$ and, respectively, $\CSB$ and $\SBV$ for the set of all CSB graphs and the set of all SBV graphs. In addition, we refer to a $2$-coloring of graph $G$ as an \emph{OSB coloring} if it is $1$-balanced  at $N(v)$ for each $v\in V(G) $ and analogously define  \emph{CSB colorings}  and \emph{SBV colorings.}

We present one additional generalization of 
of NBC and CNBC graphs  that applies only to $2$-colored graphs.  It
provides another way to include graphs that have both even and odd degree  vertices while maintaining   the concept of  $0$-balance at each vertex.

\begin{defn}\rm
  Graph $G$ is \emph{parity balanced} (PB) if  there exists a $2$-coloring of $G$ that is $0$-balanced at $N(x)$ for every even degree vertex $x$ and $0$-balanced at $N[y]$ for every odd degree vertex $y$.  We call such a coloring a \emph{PB}-coloring and denote by $\PB$ the set of all  PB graphs.
  \label{PB-def}
\end{defn}

We will see in Proposition~\ref{prop:paths} 
that paths with an even number of vertices are parity balanced while those with an odd number of vertices are not.

 For graphs in which all vertices have even degree, the definitions of PB, OSB, and NBC coincide; for graphs in which all vertices have odd degree, the definitions of PB, CSB, and CNBC coincide. We record this in the next observation.

\begin{obs} \rm 
Let $G$ be a graph.
    \begin{itemize}
           \item  If every vertex of graph $G$ has even degree then the following are equivalent:\\  (i) $G \in \PB$, (ii)\  $G \in \OSB$, (iii) $G \in \NBC$.  
           
             \item  If every vertex of graph $G$ has odd degree then the following are equivalent:\\  (i) $G \in \PB$, (ii)\  $G \in \CSB$, (iii) $G \in \CNBC$. 
    \end{itemize}
  \label{even-degree-obs}
\end{obs}

\begin{obs}
Let $G$ be a graph and $v \in V(G)$.   If $\deg(v)$ is odd then $|N(v)|$ is odd, so no $2$-coloring is $0$-balanced at $N(v)$.  If $\deg(v)$ is even then $|N[v]|$ is odd, so no $2$-coloring is $0$-balanced at $N[v]$.  

\label{zero-bal-obs}
\end{obs}

 Proposition~\ref{prop:contain} shows containments between our graph classes, and these are illustrated in the Venn diagram in Figure~\ref{fig-venn}. 
 By Lemma \ref{balance-lemma}\ref{bal-lem-pt2}, every graph  $G$ in $\SBV$  has $\beta_2(G) \le 2$ and $\beta_2[G] \le 2$.

\begin{prop}\label{prop:contain}
The following containments hold between these graph classes.  
\begin{enumerate}[label=(\roman*), font=\normalfont]
\item \label{prop:containment-01} $(\PB \cup \CSB \cup \OSB) \subseteq S\BV$ 
 \item \label{prop:containment-02}$(\NBC \cup \CNBC) \subseteq \PB \subseteq (\CSB \cap \OSB)$.
 \end{enumerate}
\end{prop}

\begin{proof} The containment in \ref{prop:containment-01} follows directly from the definitions of these classes.

If $G \in \NBC$, then every vertex has even degree; if $G \in \CNBC$, then every vertex has odd degree. Thus, the first containment in \ref{prop:containment-02} follows from Observation~\ref{even-degree-obs}.  
For the second containment in \ref{prop:containment-02}, fix a $PB$-coloring of $G$.
By  Definition~\ref{PB-def}, for a vertex $v$ of even degree we know  the coloring is $0$-balanced at $N(v)$ and hence is $1$-balanced at $N[v]$.  Similarly, for a vertex $v$ of odd degree we know  the coloring is $0$-balanced at $N[v]$ and hence is $1$-balanced at $N(v)$.
\end{proof}

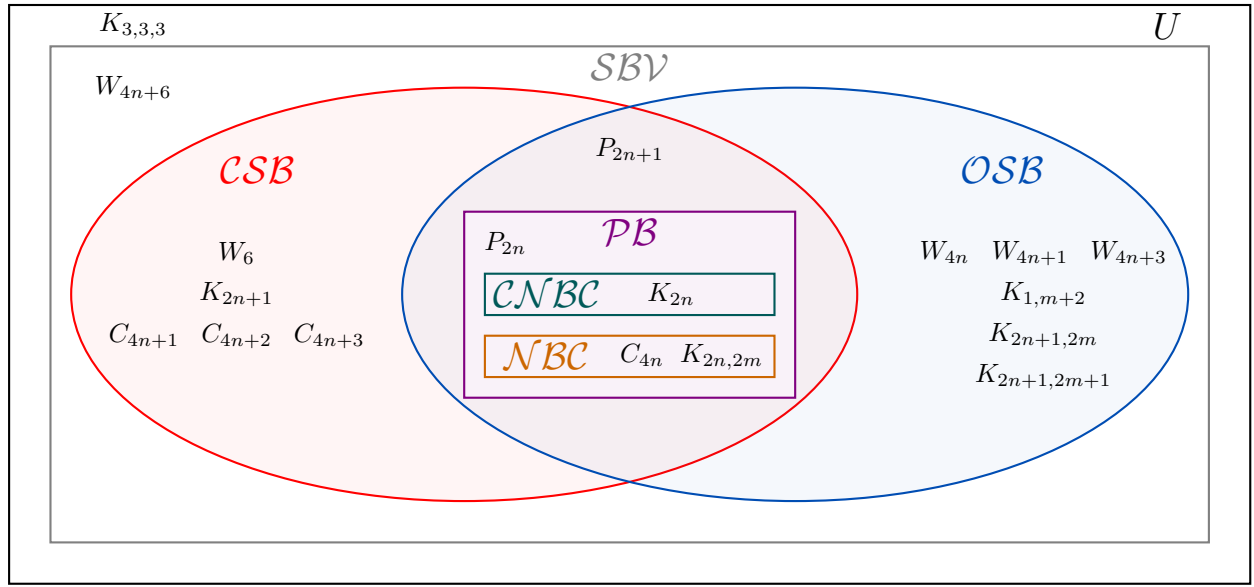
\begin{figure}[htb]
\centering
\begin{tikzpicture}[
  scale=.55,
  setbox/.style={thick},
  Ubox/.style={thick, draw=black},
  SBset/.style={setbox, draw=gray},
  CSBset/.style={setbox, draw=red},
  OSBset/.style={setbox, draw=blue!70!green},
  PBset/.style={setbox, draw=violet},
  lab/.style={font=\Large},
  expl/.style={font=\normalsize},
  NBCset/.style={thick, draw=orange!80!black},
  CNBCset/.style={thick, draw=teal!70!black}
]

\draw[Ubox] (-6,-6) rectangle (24,8);
\draw[SBset] (-5,-5) rectangle (23,7);

\draw[CSBset, fill=red, fill opacity=.04] (5,1) ellipse (9.5cm and 5cm);
\draw[OSBset, fill=blue!70!green, fill opacity=.04] (13,1) ellipse (9.5cm and 5cm);
\draw[PBset, fill=violet!5] (5,-1.5) rectangle (13,3);
\draw[NBCset]  (5.5,-1) rectangle (12.5,0);
\draw[CNBCset] (5.5,.5) rectangle (12.5,1.5);

\node[lab, text=black] at (22,7.5) {$U$};
\node[lab, text=gray] at (9,6.5) {$\SBV$};
\node[lab, text=red] at (0,4) {$\CSB$};
\node[lab, text=blue!70!green] at (18,4) {$\OSB$};
\node[lab, text=violet] at (9,2.5) {$\PB$};
\node[lab, text=orange!80!black] at (7,-.5) {$\NBC$};
\node[lab, text=teal!70!black]   at (7,1) {$\CNBC$};

\node[expl] at (-3,7.5) {$K_{3,3,3}$};

\node[expl] at (-3,6) {$W_{4n+6}$};

\node[expl] at (6,2.25){$P_{2n}$};

\node[expl] at (10,1) {$K_{2n}$};

\node[expl] at (10.5,-.5) {$C_{4n}\,\,\;K_{2n,2m}$};

\node[expl] at (9,4.5){$P_{2n+1}$};

\node[expl] at (-.5,1) {$K_{2n+1}$};
\node[expl] at (-.5,2) {$W_6$};
\node[expl] at (-.5,0)
{$C_{4n+1}\,\,\,\,\,C_{4n+2}\,\,\,\,\,C_{4n+3}$};

\node[expl] at (19,2)
{$W_{4n}\,\,\,\,\,W_{4n+1}\,\,\,\,\,W_{4n+3}$};
\node[expl] at (19,1)
{$K_{1,m+2}$};
\node[expl] at (19,0)
{$K_{2n+1,2m}$};
\node[expl] at (19,-1)
{$K_{2n+1,2m+1}$};

\end{tikzpicture}

\caption{Venn diagram of graph classes $\OSB,\CSB,\PB,\SBV$ in the universe $U$ of finite graphs with at least  three vertices and  $m,n\geq 1$. }  
\label{fig-venn}
\end{figure}

Figure~\ref{fig-venn} also contains separating examples between classes of graphs, and the  location of these separating examples  in the Venn diagram is justified by results in the remainder of this paper.
In this section, characterization results for complete graphs appear  in Proposition~\ref{complete-prop},
for complete bipartite graphs in Proposition~\ref{complete-bipartite-prop}, 
  for paths in Proposition~\ref{prop:paths}, 
  for cycles in Proposition~\ref{cycle-prop}, 
 and  for wheels in Proposition~\ref{wheel-prop}.
 The particular multipartite graph $K_{3,3,3}$ is discussed at the end of this section. 
 Section~\ref{caterpillar-sec} is dedicated to  trees and caterpillars and Section~\ref{complete-multipartite-sec}
 to complete multipartite graphs in general.
 
\begin{prop}
   If $n$ is even, then the complete graph $K_{n}$ is in $\CNBC$. If $n$ is odd and $n\ge 3$, then  $K_n$ is in $\CSB$ but not in $\OSB$.  
   \label{complete-prop}
\end{prop}

\begin{proof}
If $n$ is even,  all vertex degrees are odd and the $2$-coloring in which half the vertices are red and the other half are blue  is $0$-balanced at every closed neighborhood.  Thus in this case, $K_n \in \CNBC$. If $n$ is odd, then the coloring in which $\lceil \frac{n}{2} \rceil$ vertices are red and $\lfloor \frac{n}{2} \rfloor$ vertices are blue is a CSB coloring.  However, there is no OSB coloring of $K_n$ for $n$ odd because in any such coloring there would be more vertices of one color (say red) than the other, and then the coloring would not be $1$-balanced at $N(v)$  for a blue vertex $v$.
\end{proof}

Proposition~\ref{complete-bipartite-prop} follows from our results on complete multipartite graphs in Section~\ref{complete-multipartite-sec}, but we provide a short direct proof here.

\begin{prop}
   The complete graph bipartite graph $K_{n,m}$ is in $\OSB$ for all $n,m \ge 1$.  If both $n$ and $m$ are even then $K_{n,m} \in \NBC$.  If one or both of $n,m$ is odd and $K_{n,m} \not\in \{ K_{1,1}, K_{1,2}\}$,  then $K_{n,m} \not\in \CSB$.  
  
   \label{complete-bipartite-prop}
\end{prop}

\begin{proof}
    Let $X \cup Y$ be a bipartition of $K_{n,m}$ where $|X| = n$ and $|Y| = m$.  The coloring in which $\lfloor \frac{n}{2}\rfloor$ vertices of $X$ and $\lfloor \frac{m}{2}\rfloor$ vertices of $Y$  are red and the remaining vertices are blue shows that $K_{n,m} \in \OSB$  for  all $n,m \ge 1$ and $K_{n,m} \in \NBC $ when $n$ and $m$ are both even.   Now consider the case in which one or both of $n,m$ are odd and $K_{n,m} \not\in \{ K_{1,1}, K_{1,2}\}$.  Without loss of generality, assume $n$ is odd and if $n$ and $m$ are both odd, assume $n \le m$, so $m \ge 2$.  For a contradiction,  suppose there exists a CSB coloring of $K_{n,m}$ and assume that there are more red vertices in $X$ than blue vertices. 
    Since the coloring is $1$-balanced at $N[y]$ for each $y \in Y$, we know that all vertices in $Y$ are blue. If $m \ge 3$ then  the coloring is not $1$-balanced at $N[x]$ for any   vertex $x \in X$.  If $m=2$ then $n \ge 3$.  If $X$ contains at least three red vertices, then the coloring is not $1$-balanced at $N[y]$ for any   $y \in Y$,  and otherwise, $X$ contains a blue vertex $v$ and the coloring is not $1$-balanced at $N[v]$.  Each instance leads to a contradiction. 
\end{proof}

For $n \ge 3$, the path $P_n$ contains both even and odd degree vertices, and  therefore $P_n \not\in (\NBC \cup \CNBC)$.  Hence, the next result situates paths in the Venn diagram of Figure~\ref{fig-venn}.

\begin{prop} The path $P_n$ is in $\CSB \cap \OSB$ for all $n$, and $P_{n} \in \PB$ if and only if $n $ is even or $n=1$.
\label{prop:paths} 
\end{prop}

\begin{proof}
Label the vertices in the path $P_n$ consecutively as $v_1, v_2, \ldots , v_{n-1}, v_n$. The $2$-coloring in which $v_i$ is red if $i$ is odd and $v_i$ is blue if $i$ is even is  a CSB coloring.  The $2$-coloring in which $v_i$ is red if $i \equiv 1,4 \pmod 4$  and $v_i$ is blue if $i \equiv 2,3 \pmod 4$ is an OSB coloring for all $n$, and a PB coloring when $n$ is even.  Therefore, $P_n \in \CSB \cap \OSB$ for all $n$, and $P_n \in \PB$ when $n$ is even.

It remains to show that if $P_n \in \PB$ then $n$ is even.  Suppose that $P_{n} \in \PB$   and fix a PB coloring of $P_n$.    Without loss of generality, we may assume that  $v_1$ is red.  
Since $\deg(v_1) = 1$, the coloring is $0$-balanced at  $N[v_1]$, and  therefore $v_2$ must be blue.  Next, since $\deg(v_2) = 2$, the coloring is $0$-balanced at $N(v_2)$, so  $v_3$ must also be blue.  Continuing in this way for $1 \le j \le n$, it follows that $v_j$ has the same color as $v_{j-1}$ if $j$ is odd and the opposite color if $j$ is even.
Since vertex $v_n$ has degree $1$, the coloring is $0$-balanced at $N[v_n]$, so $v_n$  must be the opposite color from $v_{n-1}$, and therefore $n$ is even. 
\end{proof}

For cycles $C_n$, every vertex has degree 2, so 
by Observation~\ref{even-degree-obs}, we know that $C_n \in \OSB$ if and only if $C_n \in \NBC$.   Therefore, the next result situates cycles in the Venn diagram of Figure~\ref{fig-venn}. 

\begin{prop} The cycle $C_n$ is in  $\CSB$ for all $n$ and 
is in $\OSB$ if and only if $n$ is a multiple of $4$.  
\label{cycle-prop}
\end{prop}

\begin{proof} 
Let $C_n$ be the cycle with vertices labeled consecutively as  $v_1, v_2, \ldots v_n$.  To see that $C_n$ is in $\CSB$ color vertex $v_i$ red if $i$ is odd and blue if $i$ is even. As a result, there are no three consecutive vertices of the same color, so the coloring is $\CSB$.

In \cite{FM24} it is shown that 
the cycle $C_n$ is in $\NBC$ if and only if $n$ is a multiple of $4$.  Therefore, by Observation~\ref{even-degree-obs}, the cycle $C_n$ is in $\OSB$ if and only if $n$ is a multiple of $4$.
\end{proof}

\begin{prop}
Let $W_n$  be the wheel consisting of a cycle $C_n$ and a central vertex  adjacent to every vertex on the cycle.
\begin{itemize}
\item  If  $n \ge 4$ and $n \equiv 1, 3, 4 \pmod 4$ then $W_n \in \OSB$ and $W_n \not\in \CSB$

 \item If   $n >6$ and $n \equiv 2 \pmod 4$ then   $W_n \in \SBV$ and $W_n \not\in (\OSB \cup \CSB)$.
  \item  If $n=6$ then $W_6 \in \CSB$ and $W_6 \not\in \OSB$  
\end{itemize}
  
\label{wheel-prop}
\end{prop}

\begin{proof}

    Let $W_n$ be the wheel with  central vertex $v$ and the vertices  on the cycle labeled consecutively as  $v_1, v_2, \ldots v_n$.  We begin by showing  the set containments  in the statement of the proposition by constructing  $2$-colorings of $W_n$.

 First we consider $n \equiv 1, 3, 4 \pmod 4$  and provide an OSB coloring of $W_n$.  Color $v$ blue, and  for $1 \le j \le n$, color $v_j$ red if $j \equiv 1, 2 \pmod 4$ and blue if $j \equiv  3, 4 \pmod 4$.   
For $j \in \{2,3,  \ldots, n-1\}$, vertex $v_j$ has one red and one blue neighbor on the cycle as well as the blue neighbor $v$, so the coloring is $1$-balanced at $N(v_j)$.  Vertex $v_1$ has red neighbor $v_2$ and blue neighbor $v$, so the coloring is $1$-balanced at $N(v_1)$ regardless of the color of $v_n$.  Similarly, vertex $v_n$ has red neighbor $v_1$ and blue neighbor $v$, so  the coloring is $1$-balanced at $N(v_n)$.  Vertex $v$ has an equal number of red and blue neighbors if $n \equiv 4 \pmod 4$ and  one extra red neighbor if $n \equiv 1,3 \pmod 4$, so the coloring is $1$-balanced at $N(v)$. Thus our coloring of $W_n$ is an  OSB coloring and $W_n \in \OSB$.

   Next we  consider $n \equiv 2 \pmod 4$  and construct a SBV coloring for $W_n$ that is a CSB coloring when $n=6$.     Color the central vertex $v$ blue,  and color  $v_j$ red if $j \equiv 1, 2 \pmod 4$ and blue if $j \equiv  3, 4 \pmod 4$ for $1 \leq j \leq n-4$.  Color $v_{n-3}$ and $v_{n}$ blue and color $v_{n-2}$ and $v_{n-1}$ red. Observe that $v$ has $\frac{n}{2} + 1$ red neighbors and $\frac{n}{2} - 1$ blue neighbors, therefore the coloring is $1$-balanced at $N[v]$.   One can check that the coloring is also  $1$-balanced at $N(v_j)$ for $1 \le j \le n$, so it is an SBV coloring and $W_n \in \SBV$.  When $n=6$ this coloring is in fact $1$-balanced at $N[v_j]$ for $1 \le j \le n$, so it is a CSB coloring and $W_6 \in \CSB$.
   
   It remains to show the non-containments in the proposition.  We first show that if $n \equiv 2 \pmod 4$  then $W_n \not\in \OSB$.   Write $n = 4t+2$, where $t$ is a positive integer,  and for a contradiction, assume that $W_n$ is in $\OSB$.  Fix an OSB coloring of $W_n$ and, without loss of generality, we may assume that $v$ is blue.  Since the coloring is $1$-balanced at $N(v)$, we know that $N(v)$ contains  exactly $2t+1$ red vertices and $2t+1$ blue vertices.  There cannot be three consecutive blue vertices along the cycle because the middle vertex of the three would have three blue and no red vertices in its open neighborhood.  Thus, the $2t+1$ blue vertices along the cycle appear in clusters of size $1$ or $2$ vertices, and hence there are at least $t+1$ such blue clusters.  Red vertices must appear between these blue clusters.  There cannot be a single red vertex between two clusters of blue vertices because such a vertex would have three blue vertices and no red vertices in its open neighborhood.  Thus, at least two red vertices appear in each gap between clusters of the blue vertices.  Since there are at least $t+1$ clusters of blue vertices, there are at least $t+1$ gaps between them which must be filled by at least $2t+2$ red vertices.  This is a contradiction since there are only $2t+1$ red vertices.
   
  Finally, we show that if $n \ge 4$ and $W_n \in \CSB$ then $n=6$.
   Suppose that $W_n \in \CSB$ and fix a CSB coloring of $W_n$.  Without loss of generality we may assume that $v$ is blue, and since there must be both red and blue vertices among the remaining vertices, we may also  assume that $v_1$ is blue and $v_2$ is red.  
Each blue vertex on the cycle must have two red neighbors on the cycle, and each red vertex must have a red and a blue neighbor on the cycle. Thus working consecutively around the cycle starting at $v_1$ and $v_2$ we see that  $v_j$ is blue if $j \equiv 1 \pmod 3$ and $v_j$ is red if $j \equiv 2,3 \pmod 3$.  Since $v_1$ is blue  and $v_2$ is red, we know $v_n$ is red and therefore, $v_{n-1}$ is also red, so $n \equiv 3 \pmod 3$. The coloring is $1$-balanced at $N[v]$, so $n < 9$ and hence $n=6$.  
\end{proof}

In Theorem~\ref{sbv-cmp} we characterize the complete multipartite graphs that are in $\SBV$.
The graph $K_{3,3,3}$ has an odd number of vertices, no singleton parts, and $3$ non-singleton odd parts, thus  Theorem~\ref{sbv-cmp} implies that $K_{3,3,3} \not\in\SBV$. This justifies the  placement  of $K_{3,3,3}$ in   Figure~\ref{fig-venn}, and indeed  $K_{3,3,3}$ is not in any of our neighborhood $\lambda$-balanced families for $\lambda\leq 1$.

\section{Caterpillars and trees}

\label{caterpillar-sec}


In this section, we examine the class of trees beyond paths and prove in Theorem~\ref{trees-OSB-Thm} that every tree belongs to $\OSB$. Thus, for any tree $T$, we have $\beta_2(T) \le 1$ and by Lemma~\ref{balance-lemma} parts \ref{bal-lem-pt2} and \ref{bal-lem-pt4} we know $\beta_2([T]) \le \beta_2[T] \le 2$. 
Within the class $\OSB$, some trees are members of $\PB$, others are in $\CSB$ but not $\PB$, and still others are not in $\CSB$.  We analyze the possibilities completely for caterpillars in Theorems~\ref{caterpillar-pb} and \ref{CSB-caterpillar}.  Recall that a caterpillar consists of a path $P$ (sometimes called the spine) and vertices of degree $1$ adjacent to  vertices of $P$.

\begin{thm}
All trees are in $\OSB$. 
    \label{trees-OSB-Thm}
\end{thm}

\begin{proof}
     We proceed by induction on the number of vertices.  For our base case, note that $P_1 \in \OSB$.  Assume all trees on $n$ vertices are  in $\OSB$ for some $n \ge 1$.  Let $T$ be a tree on $n+1$ vertices, let $v$ be a leaf of $T$ and let $x$ be its unique neighbor.  Then $T-v$ is a tree on $n$ vertices, so by our induction hypothesis it is in $\OSB$.  Fix an OSB coloring of $T-v$.   Without loss of generality, we may assume that $N(x)$ contains  at least as many red vertices as blue in $T-v$.  Color $v$ blue.  The new coloring is $1$-balanced at $N(y)$ for all $y \in V(T)$.  Thus $T \in \OSB$ and the result follows by induction.
\end{proof}

We will see in Theorems~\ref{caterpillar-pb} and \ref{CSB-caterpillar}
that in constructing a  PB or CSB coloring  of a caterpillar $G$, there are restrictions on the colors assigned to vertices on the spine.  The next result shows that any $2$-coloring of a path $P$ can be extended to a caterpillar in $\CNBC$.  In the resulting caterpillar, the spine  may be longer than $P$.

\begin{prop}
Any $2$-coloring of a path $P$  can be extended, by adding leaves adjacent to vertices on $P$, to a $2$-coloring of a caterpillar that is $0$-balanced at every closed neighborhood.   
    \label{path-make-CNB-prop}
\end{prop}

\begin{proof}
Label the vertices of $P$ consecutively as $v_1, v_2, \ldots, v_n$ and fix a $2$-coloring of $P$.  
If $v_1$ and $v_2$ have the same color, add a leaf to $v_1$ of the opposite color, and similarly for $v_n$ and $v_{n-1}$.  For $2 \le j \le n-1$, if $v_{j-1}$, $v_j$, and $v_{j+1}$  all have the same color, add three leaves adjacent to $v_j$, each having the opposite color as $v_j$.   If $v_{j-1}$ and $v_{j+1}$ have one color and $v_j$ the opposite color, add one leaf adjacent to $v_j$ with the same color as $v_j$.  Finally, if $v_{j-1}$ and $v_{j+1}$ have opposite colors, add a leaf adjacent to $v_j$ with the opposite color as $v_j$.  This results in a $2$-coloring of a caterpillar $G$ that is $0$-balanced at every closed neighborhood. 
\end{proof}

\begin{defn}
\rm
 Let  $G$ be a caterpillar and $P$  a longest path in $G$.  For a vertex $v$ on $P$ we  define the \emph{weight} of $v$, denoted by $\wgt(v)$, to be  the number of neighbors of $v$ that are not on $P$.

\end{defn}

\begin{lem}
 Let  $G$ be a   caterpillar with a PB coloring and $P$  a longest path in $G$.  If $w$ is a non-endpoint vertex of $P$ with $\wgt(w) \in \{2,3\}$ then $w$ and its two neighbors on $P$ all have the same color and if  
    $\wgt(w) \in \{0,1\}$ then the two neighbors of $w$ on $P$  have opposite colors.
    \label{caterpillar-lemma}
\end{lem}

\begin{proof}
  Let $w$ be a non-endpoint vertex of $P$ and without loss of generality,  assume $w$ is blue.     If $\wgt(w) = 2$ then  $\deg(w) = 4$ and the two leaf neighbors of $w$ are red, so the two neighbors of $w$ on $P$ must be blue since the coloring is $0$-balanced at $N(w)$. If  $\wgt(w) = 3$ then  $\deg(w) = 5$ and the three leaf neighbors of $w$ are red, so the two neighbors of $w$ on $P$ must be blue since the coloring is $0$-balanced at  $N[w]$. 
Similarly, if $\wgt(w) = 0$ then  $\deg(w) = 2$ and both neighbors are on $P$, so they must be opposite colors and if  $\wgt(w) = 1$ then  $\deg(w) = 3$ and  its leaf neighbor not on $P$ is red, so  the two neighbors   on $P$ must be opposite colors.
\end{proof}

\begin{thm}
Let $G$ be a caterpillar, let $P$ be a longest path in $G$ and let $\{w_1, w_2, \ldots, w_m\}$ be the  set of  vertices  on $P$ that have weight $2$ or $3$.  
Then $G \in \PB $ if and only if  $\wgt(v) \le 3$ for all vertices $v$ on $P$ and the path segments remaining when the vertices $w_1, w_2, \ldots, w_m$ are removed from $P$ each have  an even number of vertices.
\label{caterpillar-pb}
\end{thm}

\begin{proof}
Let $Q_0, Q_1, \ldots, Q_m$ be the path segments remaining when the vertices $w_1, w_2, \ldots, w_m$ are removed from $P$, so that $P$ consists of $Q_0$ followed by $w_1$, followed by $Q_1$, followed by $w_2$, etc. and ending with $Q_m$.

First suppose that every vertex $v$ on $P$ has $\wgt(v) \le 3$ and that each  path segment $Q_i$ has an  even number of vertices for $0 \le i \le m$.
 We provide a PB coloring for $G$ as follows.  
 We color the vertices of $P$ from left to right and when we get to each   $w_i$ we assign it the same  color as the vertex proceeding it on $P$.
  Next we  consider path segment $Q_i$ consisting of  $j(i)$ vertices, where $j(i)$ is even, and label  the vertices consecutively as $x^i_1, x^i_2, \ldots, x^i_{j(i)}$.  If $i=0$, arbitrarily choose color  blue for $x^0_1$ (the leftmost vertex of $P$).  Otherwise,  assign $x^i_1$   the same color as $w_i$.  For $\ell \ge 2$, give vertex $x^i_{\ell}$ the same color as $x^i_1$ if $\ell \equiv 0,1 \pmod 4$, and otherwise, assign it the opposite color. This assigns a color to every vertex of $P$.  Note that  every vertex $v$  of $P$ with $\deg(v) \ge 2$ and $\wgt(v) \in \{0,1\}$
 has a red neighbor and a blue neighbor on $P$. 
The  vertices of $G$ that are still uncolored are leaves, and each is assigned the opposite color from that of its neighbor, so the  coloring is $0$-balanced at the closed neighborhoods of these leaves.   

We next verify that our coloring is parity balanced at the vertices on $P$.   Observe that the color assigned to each vertex $w_i$ is the same as the colors assigned to its two neighbors on $P$.  If $\wgt(w_i) = 3$ then $\deg(w_i) = 5$   and the coloring is $0$-balanced at  $N[w_i]$, and  if $\wgt(w_i) = 2$ then $\deg(w_i) = 4$   and the coloring is $0$-balanced at  $N(w_i)$.    The endpoints of $P$ each have degree $1$ and are assigned the opposite color from that of their neighbor, thus the coloring is parity balanced at those vertices.  The remaining vertices $v$  in the segments $Q_i$ each have a red neighbor and a blue neighbor on $P$ and $\wgt(v) \in \{0,1\}$, so the coloring is $0$-balanced at $N(v)$   if $\wgt(v) = 0$  and $0$-balanced at $N[v]$ if $\wgt(v) = 1$.   This completes the proof that our coloring is parity balanced.

Conversely, suppose that $G \in \PB$ and fix a PB coloring of $G$.
 We first show that $\wgt(v) \le 3$ for every vertex $v$ on $P$.  For a contradiction, suppose that there exists a vertex $v$ on $P$ with $\wgt(v) \ge 4$, and without loss of generality assume that $v$ is blue.  Each leaf neighbor of $v$ must be red, so $v$ has at least $4$ red neighbors and at most $2$ blue neighbors (those on $P$), contradicting the asssumption that the coloring was parity balanced.

It remains to show that each $Q_i$ has an even number of vertices.  Suppose for a contradiction that there exists an $i: 0 \le i \le m$ for which $Q_i$ has an odd number of vertices and label the vertices from left to right as  $y^i_1, y^i_2, \ldots, y^i_{2j+1}$.  Without loss of generality we may assume that $y^i_1$ is blue.  If $i=0$ then $y^0_1$ is a leaf and $y^0_2$ must be red.  Otherwise, $i \ge 1$ and $y^i_1$ is adjacent to $w_i$, so by Lemma~\ref{caterpillar-lemma} it must be the same color as $w_i$, so both $w_i$ and $y^i_1$ are blue.  Now $\wgt(y^i_1) \in \{0,1\}$, so by Lemma~\ref{caterpillar-lemma} we know $y^i_2$ is red.  We continue applying Lemma~\ref{caterpillar-lemma} to conclude that $y^i_3$ is red, $y^i_4$ is blue, $y^i_5$ is blue, etc. so that $y^i_{\ell}$ is blue if $\ell \equiv 0,1 \pmod 4$ and  red if $\ell \equiv 2,3 \pmod 4$.  In particular, $y^i_{2j}$ and $y^i_{2j+1}$ have the same color.  If $y^i_{2j+1}$  is an endpoint of $P$, this is a contradiction because  in this case $y^i_{2j+1}$ would be a leaf and must be the  opposite color from its neighbor. Otherwise, $y^i_{2j+1}$  is adjacent to $w_{i+1}$.  
Applying Lemma~\ref{caterpillar-lemma}, the two neighbors of $y^i_{2j+1}$ on $P$ must have opposite colors, so $w_{i+1}$ would get the color opposite that of $y^i_{2j}$ and $y^i_{2j+1}$.  However, the same lemma  applied to $w_{i+1}$ implies that $w_{i+1}$ and $y^i_{2j+1}$ get the same color, a contradiction.
\end{proof}

\begin{lem}
 Let  $G$ be a   caterpillar with a CSB coloring and let  $P$  a longest path in $G$.  The following hold for any non-endpoint vertex $v$ of $P$.
 \begin{itemize}
\item  If $\wgt(v) \in \{3,4\}$ then $v$ and  its two neighbors on $P$ all have the same color.

\item If $\wgt(v) \in \{0,1\}$ then
the neighbors of $v$ on $P$ cannot both be $v$'s color.  

\item If $\wgt(v) \in \{1,2\}$ then the two neighbors of $v$ on $P$ cannot both be the opposite color from $v$.

     \end{itemize}
 
    \label{caterpillar-lemma-2}
\end{lem}

\begin{proof}
  Without loss of generality, we may assume $v$ is blue.     If $\wgt(v) \in\{3,4\}$ then  $v$ has three or four leaf neighbors that are red, and therefore, both is neighbors on $P$ must be blue.  If $\wgt(v) = 1$ then $v$ has one red leaf neighbor and $v$ is blue, so  among its two neighbors on $P$, one must be blue and the other red.  If $\wgt(v) = 0$ then $\deg(v) = 2$ and since $v$ is blue, its neighbors on $v$  cannot both be blue. Finally, if $\wgt(v) = 2$ then $v$ has two red leaf neighbors and $\deg(v) = 4$, so its neighbors on $P$ cannot both be red.
\end{proof}

The next theorem characterizes caterpillars that have a CSB-coloring.
Our proof involves a coloring of paths defined as follows:  if  path $Q$ consisists of the consecutive vertices $z_1, z_2, \ldots, z_m$,  then  a \emph{double-alternating coloring}  of $Q$ is an assignment of red or blue to each vertex of $Q$ so that $z_i$ has the same color as $z_1$ if $i \equiv 0,1 \pmod 4$ and the opposite color otherwise.

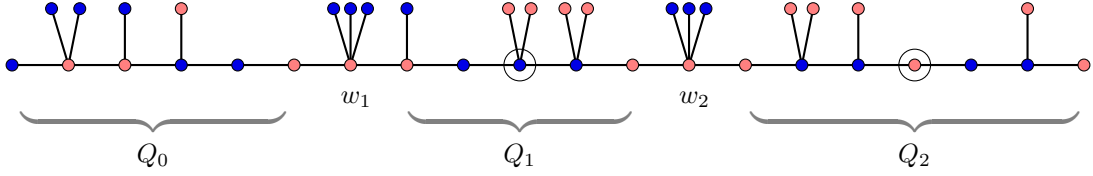
\begin{figure}\centering 
\begin{tikzpicture}[scale=.75]

\draw[thick] (0,0)--(19,0);
\draw[thick] (.7,1)--(1,0)--(1.2,1);
\draw[thick] (2,0)--(2,1);
\draw[thick] (3,0)--(3,1);
\draw[thick] (5.7,1)--(6,0)--(6.3,1);
\draw[thick] (6,0)--(6,1);
\draw[thick] (7,0)--(7,1);
\draw[thick] (8.8,1)--(9,0)--(9.2,1);
\draw[thick] (9.8,1)--(10,0)--(10.2,1);
\draw[thick] (11.7,1)--(12,0)--(12.3,1);
\draw[thick] (12,0)--(12,1);
\draw[thick] (13.8,1)--(14,0)--(14.2,1);
\draw[thick] (15,0)--(15,1);
\draw[thick] (18,0)--(18,1);

\draw(9,0) circle [radius=8pt] ;
\draw(16,0) circle [radius=8pt] ;

\filldraw[blue!90!black, draw=black]
(0,0)circle [radius=3pt]
(3,0)circle [radius=3pt]
(4,0)circle [radius=3pt]
(8,0)circle [radius=3pt]
(9,0)circle [radius=3pt]
(10,0)circle [radius=3pt]
(14,0)circle [radius=3pt]
(15,0)circle [radius=3pt]
(17,0)circle [radius=3pt]
(18,0)circle [radius=3pt]

(.7,1)circle [radius=3pt]
(1.2,1)circle [radius=3pt]
(2,1)circle [radius=3pt]
(6,1)circle [radius=3pt]
(5.7,1)circle [radius=3pt]
(6.3,1)circle [radius=3pt]
(7,1)circle [radius=3pt]
(12,1)circle [radius=3pt]
(11.7,1)circle [radius=3pt]
(12.3,1)circle [radius=3pt]
;

\filldraw[red!50, draw=black]
(1,0)circle [radius=3pt]
(2,0)circle [radius=3pt]
(5,0)circle [radius=3pt]
(6,0)circle [radius=3pt]
(7,0)circle [radius=3pt]
(11,0)circle [radius=3pt]
(12,0)circle [radius=3pt]
(13,0)circle [radius=3pt]
(19,0)circle [radius=3pt]

(3,1)circle [radius=3pt]
(9.2,1)circle [radius=3pt]
(8.8,1)circle [radius=3pt]
(10.2,1)circle [radius=3pt]
(9.8,1)circle [radius=3pt]
(14.2,1)circle [radius=3pt]
(13.8,1)circle [radius=3pt]
(15,1)circle [radius=3pt]
(18,1)circle [radius=3pt]
;

\filldraw[red!50, draw=black]
(16,0)circle [radius=3pt]
;

\node(1) at (6.1,-.6) {$w_1$};

\node(1) at (12.1,-.6) {$w_2$};

\node(2) at (2.5,-1) {\Large \color{gray} $\underbrace{\;\;\;\;\;\;\;\;\;\;\;\;\;\;\;\;\;\;\;\;\;\;\;\;\;}$};

\node(3) at (2.5,-1.6) {$Q_0$};

\node(2) at (9,-1) {\Large \color{gray} $\underbrace{\;\;\;\;\;\;\;\;\;\;\;\;\;\;\;\;\;\;\;\;\;}$};

\node(3) at (9,-1.6) {$Q_1$};

\node(2) at (16,-1) {\Large \color{gray} $\underbrace{\;\;\;\;\;\;\;\;\;\;\;\;\;\;\;\;\;\;\;\;\;\;\;\;\;\;\;\;\;\;\;}$};

\node(3) at (16,-1.6) {$Q_2$};

\end{tikzpicture}

\caption{A caterpillar with a CSB-coloring. As described  the proof of Theorem~\ref{CSB-caterpillar}, the circled vertices  and vertices $w_1$ and $w_2$ interrupt the double-alternating coloring.
}
 \label{caterpillar-figure}
\end{figure}

\begin{thm}
Let $G$ be a caterpillar, let $P$ be a longest path in $G$, and let $\{w_1, w_2, \ldots, w_m\}$ be the set of vertices on $P$ that have weight $3$ or $4$. 
Then $G \in \CSB$ if and only if $\wgt(v) \le 4$ for all vertices $v$ on $P$,  and for $0 \le i \le m$,  the path segments $Q_i$ remaining when the vertices $w_1, w_2, \ldots, w_m$ are removed from $P$ satisfy at least one of the following three conditions:

\begin{enumerate}
    \item[\rm{(a)}] $Q_i$ contains an even number of vertices.
    \item[\rm{(b)}] There exists a weight $2$ vertex in an odd position of $Q_i$.
    \item[\rm{(c)}] There exists a  weight $0$ vertex in an even position of $Q_i$.
\end{enumerate}
\label{CSB-caterpillar}
\end{thm}

\begin{proof}
Let $Q_0, Q_1, \ldots, Q_m$ be the path segments remaining when the vertices $w_1, w_2, \ldots, w_m$ are removed from $P$, so that $P$ consists of $Q_0$ followed by $w_1$, followed by $Q_1$, followed by $w_2$, etc. and ending with $Q_m$.

First suppose that every vertex $v$ on $P$ has $\wgt(v) \le 4$ and that each path segment $Q_i$ satisfies one of the three listed conditions in the theorem.  We provide a CSB-coloring for $G$ and illustrate the coloring in   Figure~\ref{caterpillar-figure}.

We focus on coloring the vertices of the spine, keeping in mind that  each leaf must be assigned the color opposite that of its neighbor.
 We color the vertices of $P$ from left to right and when we get to each $w_i$ we assign it the same color as the vertex proceeding it on $P$. Next we consider path segment $Q_i$ consisting of $j(i)$  vertices and label the vertices consecutively as $x^i_1, x^i_2, \ldots, x^i_{j(i)}$.
 If $i=0$, arbitrarily choose color  blue for $x^0_1$ (the leftmost vertex of $P$).  Otherwise,  assign $x_1^i$   the same color as $w_i$.   We consider each of the three cases (a),(b),(c) separately.  If $j(i)$ is even (case (a)), use a double-alternating coloring on $Q_i$, as illustrated in the coloring of $Q_0$ of Figure~\ref{caterpillar-figure}. 
  Otherwise, we may assume that $j(i)$ is odd.   In Case (b), let $\ell$ be an odd integer with $1 < \ell < j(i)$ for which $\wgt(x^i_{\ell}) = 2$.  Assign a double-alternating coloring to the path segment from $x^i_1$ to $x^i_{\ell-1}$, assign $x^i_{\ell}$ and $x^i_{\ell+1}$
 the same color as $x^i_{\ell-1}$, and then assign a double-alternating coloring to the  path segment from $x^i_{\ell+1}$ to $x^i_{j(i)}$. This is  illustrated in the coloring of $Q_1$ in Figure~\ref{caterpillar-figure}. In Case (c), let $\ell$ be an even integer with $1 < \ell < j(i)$ for which $\wgt(x^i_{\ell}) = 0$.  Assign a double-alternating coloring to the path segment from $x^i_1$ to $x^i_{\ell-1}$, assign $x^i_{\ell}$ the opposite color from $x^i_{\ell-1}$ and then assign a double alternating coloring to the  path segment from $x^i_{\ell}$ to $x^i_{j(i)}$.   This is  illustrated in the coloring of $Q_2$ in Figure~\ref{caterpillar-figure}. Now we have assigned a color to every vertex on $P$. The  vertices of $G$ that are still uncolored are leaves, and each is assigned the opposite color from that of its neighbor.

 We next verify that  this is a CSB  coloring.  Each leaf is colored the opposite color from its neighbor, so  the coloring is $0$-balanced at $N[v]$  for each leaf $v$.    By construction, each vertex $w_i$ has the same color as its two neighbors on $P$ and the opposite color from its $3$ or $4$ leaf neighbors, therefore  the coloring is $1$-balanced at $N[w_i]$  for $i: 1\le i \le m$.
The remaining vertices are on $Q_i$ for some $i: 1 \le i \le m$.  The vertices  $v$ colored using a double-alternating coloring each have one blue neighbor and one red neighbor on $P$ and also have $\wgt(v) \in \{0,1,2\}.$ Thus the coloring is $1$-balanced at $N[v]$.  The only remaining vertices  arise  either as  a vertex $x$ in an odd position of $Q_i$ with $\wgt(x) = 2$ in case (b)  or a vertex $y$ in an even position of $Q_i$ with $\wgt(y) = 0$ in case (c).  By construction, such a vertex $x$ will have the same color as its two neighbors on $P$, and the opposite color from its two leaf neighbors, so the coloring is $1$-balanced at $N[x]$.  Similarly, vertex $y$ will have the opposite color from its two neighbors on $P$  and no leaf neighbors, so the coloring is also $1$-balanced at $N[y]$.  Thus our coloring is a CSB coloring.

Conversely, suppose that $G$ is $1$-balanced at all closed neighborhoods and fix a CSB coloring of $G$.  We first show that $\wgt(v) \le 4$ for every vertex $v$ on $P$.  For a contradiction, suppose that there exists a vertex $v$ on $P$ with $\wgt(v) \ge 5$ and without loss of generality, assume that $v$ is blue.  Each leaf neighbor of $v$ must be red, so $v$ has at least $5$ red neighbors  and at most $3$ blue vertices in its closed neighborhood, a contradiction.

It remains to show that for each $Q_i$, at least one of (a),(b), (c) holds, so for a contradiction, suppose  there exists an $i: 0 \le i \le m$ for which  none of them hold.  Label the vertices  on $Q_i$ from left to right as  $y^i_1, y^i_2, \ldots, y^i_{2j+1}$.   Thus $\wgt(y^i_{\ell} ) \in \{0,1\}$ for $\ell$ odd and $\wgt(y^i_{\ell} ) \in \{1,2\}$ for $\ell$ even.
Without loss of generality we may assume that $y^i_1$ is blue.  If $i=0$ then $y^0_1$ is a leaf and  $y^0_2$ must be red.  Otherwise, $i \ge 1$ and $y^i_1$ is adjacent to $w_i$, and $\wgt(w_i) \in \{3,4\}$. We know $y^i_1$ is the same color as $w_i$ by Lemma~\ref{caterpillar-lemma-2},  so both $w_i$ and $y^i_1$ are blue.  Now $\wgt(y^i_1) \in \{0,1\}$, so by Lemma~\ref{caterpillar-lemma-2} we know $y^i_2$ is red. Next  $\wgt(y^i_2) \in \{1,2\}$, so by Lemma~\ref{caterpillar-lemma-2} we know  $y^i_3$ is red.  We continue applying Lemma~\ref{caterpillar-lemma-2} to conclude that  $y^i_4$ is blue, $y^i_5$ is blue, $y^i_6$ is red, etc. so that $y^i_{\ell}$ is blue if $\ell \equiv 0,1 \pmod 4$ and  red if $\ell \equiv 2,3 \pmod 4$.  In particular, $y^i_{2j}$ and $y^i_{2j+1}$ have the same color.  If $y^i_{2j+1}$  is an endpoint of $P$, this is a contradiction because  in this case $y^i_{2j+1}$ would be a leaf and must be the  opposite color from its neighbor. Otherwise, $y^i_{2j+1}$  is adjacent to $w_{i+1}$.  Applying Lemma~\ref{caterpillar-lemma-2}, the two neighbors of $y^i_{2j+1}$ on $P$ must have opposite colors, so $w_{i+1}$ would get the color opposite that of $y^i_{2j}$ and $y^i_{2j+1}$.  However, the same lemma  applied to $w_{i+1}$ implies that $w_{i+1}$ and $y^i_{2j+1}$ get the same color, a contradiction.
\end{proof}

In Theorem~\ref{CSB-caterpillar}, we characterized those caterpillars that have a $CSB$ coloring, and we next count how many  $2$-colored caterpillars there are  when the coloring is a CSB coloring and   the spine length is fixed. We consider the spine of a caterpillar  $T$ to be a
longest path in $T$, thus the endpoints of the spine have degree $1$.  

\begin{defn}
    Let ${\mathcal A}(n)$ be the set of $2$-colored caterpillars satisfying  all of the following:
    
    {\rm(i) } the spine contains $n$ vertices,

     {\rm(ii) } the coloring is a CSB coloring, and 

    {\rm(iii)} the color of the  leftmost vertex  of the spine is specified.

    \noindent
    Furthermore, let $A(n) = |{\mathcal A}(n)|.$
    \label{An-def}
\end{defn}

We will find explicit and recursive formulas for $A(n)$.  It will be helpful to define the related class  ${\mathcal B}(n)$ of $2$-colored caterpillars.

\begin{defn}
    Let ${\mathcal B}(n)$ be the set of $2$-colored caterpillars satisfying  all of the following:
    
    {\rm(i) } the spine contains $n$ vertices,

     {\rm(ii) }the coloring is $1$-balanced at every closed neighborhood except for the leftmost vertex of the spine,  and 

    {\rm(iii)} the leftmost  two vertices  of the spine have the same specified color.

    \noindent
    Furthermore, let $B(n) = |{\mathcal B}(n)|.$
     \label{Bn-def}
\end{defn}

\begin{thm}
 The functions $A(n)$ and $B(n)$  from Definitions~\ref{An-def}  and \ref{Bn-def} satisfy the initial conditions $A(2) = A(3) = 1$,  $B(2) = 0$, and $B(3) =3$, and  the  following recurrences hold for  for  $n\ge 4$: 
 
{\rm  (i)} $A(n) = A(n-1) + 3B(n-1)$
 
 {\rm  (ii)} $B(n) = 3A(n-1) + 3B(n-1)$
 
{\rm  (iii)} $A(n)=4A(n-1)+6A(n-2)$

{\rm  (iv)} $B(n)=4B(n-1)+6B(n-2)$.

 \label{A_n-Thm}
\end{thm}

\begin{proof}

From the definitions, one can   verify that the initial conditions are satisfied.

First we show $A(n) = A(n-1) + 3B(n-1)$ for $n\ge 4$.     Let the $2$-colored caterpillar $T$  be an element of ${\mathcal A}(n)$, where  $c$  is the coloring and  $v_1$, the leftmost vertex of the spine, is red.     Label the vertices of the spine consecutively as $v_1, v_2, \ldots, v_n$.  Since the coloring is a CSB coloring, it is $1$-balanced at $N[v_1]$, so $v_2$ is blue.   We consider two cases, depending on the color of $v_3$.  
If $v_3$ is red then $\wgt(v_2) = 0$ since the coloring is $1$-balanced at $N[v_2]$, and similarly, if $v_3$ is blue, then $\wgt(v_2) \in \{0,1,2\}$.   Let $T'$ be the $2$-colored caterpillar formed by removing $v_1$ and any additional leaf neighbors of $v_2$ from $T$ and maintaining the vertex colors from $c$.  

If $v_3$ is red then  $ T' \in {\mathcal A}(n-1)$ since the coloring remains $1$-balanced at $N[v_2]$ when $v_1$ is removed.  Indeed,  the function that maps $T$ to $T'$ is  a one-to-one correspondence between elements of ${\mathcal A}(n)$ whose first three spine vertices have colors R, B, R and  elements of ${\mathcal A}(n-1)$.  
If $v_3$ is blue  then  $ T' \in {\mathcal B}(n-1)$.  Indeed,  each element of ${\mathcal B}(n-1)$ corresponds to three elements of   ${\mathcal A}(n)$, one for each possible value of  
    $\wgt(v_2)$.  Thus $A(n) = A(n-1) + 3B(n-1)$.

    Next we show $B(n) = 3A(n-1) + 3B(n-1)$ for $n \ge 4$.  
Now we let  the $2$-colored caterpillar $T$  be an element of ${\mathcal B}(n)$, where  $c'$  is the coloring and the leftmost two vertices of the spine are blue.   
 Label the vertices of the spine consecutively as $v_1, v_2, \ldots, v_n$, hence $v_1$  and $v_2$ are blue.  By the definition of  ${\mathcal B}(n)$ we know that $c'$ is $1$-balanced at $N[v]$ for every vertex of $T$ except for $v_1$.  We consider two cases depending on the color of $v_3$.  

 If $v_3$ is blue then we know $\wgt(v_2) \in \{2,3,4\}$  since   $c'$ is $1$-balanced at $N[v_2]$, and similarly, if $v_3$ is red then $\wgt(v_2) \in \{0,1,2\}$.  Let $T'$ be the $2$-colored caterpillar formed by removing $v_1$ and any additional leaf neighbors of $v_2$ from $T$ and maintaining the vertex colors from $c'$.    If $v_3$ is blue then $T' \in {\mathcal B}(n-1)$ and indeed there are three elements of ${\mathcal B}(n)$ corresponding to $T'$, one for each possible value of $\wgt(v_2)$.  Similarly, if $v_3$ is red then $T' \in {\mathcal A}(n-1)$ and  there are three elements of ${\mathcal B}(n)$ corresponding to $T'$, one for each possible value of $\wgt(v_2)$.  Thus $B(n) = 3A(n-1) + 3B(n-1)$.

 Now we combine the two recurrences:  $A(n)=A(n-1)+3B(n-1)$,  and $B(n)=3A(n-1)+3B(n-1)$ for $n\ge 4$ to finish the proof of the theorem.    
From  the first recurrence,  we conclude $B(n-1) =  \frac{1}{3}[A(n)-A(n-1)]$ for $n\ge 3$, and 
replacing $n-1$ by $n$ we get  and $B(n) = \frac{1}{3}[A(n+1)-A(n)]$. 
Combine these with 
 the second recurrence to get 
 $\frac{1}{3}[A(n+1) - A(n)] = 
 3A(n - 1)+ [A(n) - A(n - 1)] = A(n) + 2A(n-1)$. 
 This simplifies
to $A(n + 1) = 4 A(n) + 6A(n-1)$ 
for $n \ge 3$, so 
$A(n ) = 4A(n-1) + 6A(n-2)$ for $n \ge 4$. Similar calculations show that 
$B(n ) = 4B(n-1) + 6B(n-2)$ for $n \ge 4$.
\end{proof}

Using Theorem~\ref{A_n-Thm} and the method of rational generating functions (see, for instance, \cite[Chap. 4]{Stanley-EC1}), since $1-4x-6x^2=(1-(2-\sqrt{10})x)(1-(2+\sqrt{10})x)$, we find the following explicit formula for $n \ge 3$. 

\[A(n) = (-5/24 -\sqrt{10}/15) (2-\sqrt{10})^{n+1} + (-5/24 +\sqrt{10}/15) (2+\sqrt{10})^{n+1}.\] 

The values of $A(n)$ for $2\leq n \leq 10$ are $1, 1, 10, 46, 244, 1252, 6472, 33400, 172432$. 
This sequence, with the first 1 removed, is in the On-line Encyclopedia of Integer Sequences (OEIS) as A138041, and $A(n+2)$ is the number in the top row and first column and $B(n+2)$ is the number in the bottom row, first column of the matrix 
${\begin{bmatrix}
    1 & 3\\
    3 & 3
\end{bmatrix}}^n$.
This matrix arises from our recursive formulas in Theorem~\ref{A_n-Thm} since 

\begin{center}
$\begin{bmatrix}
    A(n)\\
   B(n)
\end{bmatrix}
= 
{\begin{bmatrix}
    1 & 3\\
    3 & 3
\end{bmatrix}}^{n-3} \begin{bmatrix}
   A(3)\\
 B(3)
\end{bmatrix}
$ for $n \ge 4$.
\end{center}

Thus, explicit formulas for $A(n)$ and $B(n)$ can also be found by diagonalizing the matrix. The values of $B(n)$ for $2\leq n\leq 8$ are $0,3,12, 66, 336, 1740,8976$ which are all divisible by three. The sequence $\frac{1}{3}B(n)$ appears in the OEIS as A085939.

\section{Red-blue removal and complete multipartite graphs}
\label{complete-multipartite-sec}

In this section, we  introduce the operation of red-blue removal which can be used to transform a $2$-coloring of one graph to a $2$-coloring of a simpler graph while maintaining balance properties of the coloring. We apply this technique to complete multipartite graphs and characterize those that are members of the classes $\OSB$, $\CSB$, $\SBV$ and $\PB$.
As a consequence, we prove that all complete multipartite graphs $G$ satisfy $\beta_2([G]) \leq 2$, $\beta_2(G) \le 2$ and  $\beta_2[G] \le 3$.

We introduce terms that will be used throughout this section.  In a complete multipartite graph, we call a vertex in a part of size $1$ a \emph{singleton} and we say that a part of the coloring is \emph{monochromatic} if all vertices in the part have the same color.  Additionally, a set of vertices containing more red than blue vertices is called \emph{red-heavy} and a $2$-coloring is \emph{red-heavy} if it contains more red than blue vertices.  Blue-heavy is defined similarly.
 
\begin{defn}
\rm
    For a $2$-colored  graph $G$, a \emph{red-blue removal} consists of removing a red vertex $x$  and a blue vertex $y$  for which $N(x) = N(y)$. The \emph{reduced graph} $\ghat$ is obtained by iteratively applying red-blue removals until none remain possible.
    \label{reduced-graph-def}
\end{defn}

One can verify that red-blue removals can be applied in any order, so the reduced graph  is well-defined.  For $2$-colored complete multiparite graphs, a  red-blue removal consists of removing a red and a blue vertex from one of the parts.
We record a few consequences of Definition~\ref{reduced-graph-def} in the following.

\begin{obs}
    Let $G$ be a $2$-colored complete multipartite graph and $\ghat$ the resulting reduced graph.

\begin{enumerate}

\item All parts of   $\ghat$ are monochromatic. 

\item  Odd parts of $G$ remain odd parts of $\ghat.$

\item Even parts of $G$ either remain even parts of $\ghat$ or are eliminated.

\item An SBV coloring of $G$ induces an SBV coloring of any graph obtained from $G$  by a sequence of red-blue removals.  Similar results are also true for any of OSB, CSB, and PB colorings. 

\end{enumerate}

\label{reduction-obs}
\end{obs}

For OSB graphs, the converse of the result in Observation~\ref{reduction-obs} (4)  follows from Theorem~\ref{osb-multipartite-thm} along with Observation~\ref{reduction-obs} (2). 
 That is, if a 2-coloring obtained from $G$ by a sequence of red-blue removals is an OSB coloring, then the  original coloring of $G$ is an OSB coloring. 

However, for SBV, CSB and PB colorings, the converse of the result in Observation~\ref{reduction-obs} (4) is not true.
 A $2$-coloring of $G$ that is not an SBV coloring might  be an SBV coloring when restricted to the vertices of $\ghat$.  Similarly, for CSB and PB colorings.  To see that the converse does not hold for an SBV coloring consider the $2$-coloring of $K_{3,3,3}$ in which two parts have two red and one blue vertex and one part has one red and two blue vertices.  This is not an SBV coloring, however, the reduced graph is $K_{1,1,1}$ and the coloring with two red and one blue vertices is an SBV coloring of $K_{1,1,1}$.  For CSB and PB colorings consider the graph $K_{3,3,2}$ where one part of size three has an extra red, the other has an extra blue, and the even part has one red and one blue vertex.  This is not a PB or CSB coloring of $G$, but it is when we reduce to $\ghat$ which is $K_2$.

We take advantage of  Observation~\ref{reduction-obs} in the next proofs.

\begin{thm}
A complete multipartite graph $G$ is in $ \OSB$  if and only if  $|V(G)|$ is even or  $G$  has exactly one odd part.
In particular, all complete bipartite graphs are  in $\OSB$.

\label{osb-multipartite-thm}
\end{thm}

\begin{proof}
    To prove the  backward direction, we provide  an $\OSB$ coloring of $G$.
For each even part, color half of the vertices red and the other half blue.  If there is exactly one odd part, color it with one more red than blue.  In this case, every vertex in an even part has one extra red neighbor, and every vertex in the odd part has an equal number of red and blue neighbors, so this is an OSB coloring.   Otherwise, $|V(G)|$ is even and hence
 there are an even number of odd parts.   Pair the odd parts and for each pair, color one so that it has one more red than blue vertex and the other so that it has one more blue than red.
Now, any vertex in an even part has an equal number of red and blue neighbors,  any vertex in a red-heavy odd part has one  extra blue neighbor while any vertex in a blue-heavy odd part has one extra red neighbor. Thus, again we have an OSB coloring and  $G\in \OSB$. 

Conversely, let $G$ be a complete multipartite graph with an OSB coloring.  If $|V(G)|$ is even, the first condition is satisfied, so we may assume $|V(G)|$ is odd.  Let    $\ghat$ be the reduced graph, thus $|V(\ghat)|$ is odd and must contain an odd part $P$.  By Observation~\ref{reduction-obs}(1), we know $P$ is monochromatic, and without loss of generality, we may assume that $P$ is monochromatic red. Let $R$ be the set of red vertices and $B$ be the set of blue vertices in $\ghat$.  By Observation~\ref{reduction-obs}(4), we know that our coloring is a OSB coloring of $\ghat$.   For $x \in P$ we know $|N(x)|$ is even and thus $N(x)$ consists of $s$ red and $s$ blue vertices for some integer $s \ge 0$.  Hence $|R| = s + |P|  \ge s+1$ and $|B| = s$.  If $z$ is a blue vertex in $\ghat$ then $N(z)$ contains at least $s+1$ red vertices and at most $s-1$ blue vertices, a contradiction since our coloring of $\ghat$ is OSB.  Hence there are no blue vertices in $\ghat$ which means that $s=0$ and $\ghat$ consists of a single odd part $P$.  By Observation~\ref{reduction-obs}, the graph $G$ has exactly one odd part, as desired.
\end{proof}

We next focus on characterizing the complete multipartite graphs in  $\SBV$.

\begin{lem}
If $G$ is a complete multipartite graph with a red-heavy  SBV coloring then every blue singleton in the reduced graph $\ghat$ must be a singleton in $G$. 

\label{reduce-lemma}
\end{lem}

\begin{proof}
   Let $G$ be a complete multipartite graph with a red-heavy  SBV coloring and let $\ghat$ be its reduced graph.  Let $R$ be the set of red vertices in $\ghat$ and $B$ the set of blue vertices, thus $|R| > |B|$.    Suppose that $z$ is a blue singleton in $\ghat$ and for a contradiction, suppose that $z$ is not a singleton in the original graph $G$.  Let $G'$ be the $2$-colored graph arising from adding one red and one blue vertex to $z$'s part of $\ghat$.  Note that $G'$ can be obtained from $G$ by a sequence of red-blue removals, so the $2$-coloring of $G'$ is an SBV coloring.   Let $w$ be the red vertex in $z$'s part of $G'$.  Then $N(w)$ contains  $|B|-1$  blue vertices and $|R|$ red vertices and $N[w]$ contains  $|B|-1$ blue vertices and $|R|+1$ red vertices, contradicting the fact that the coloring is an SBV coloring. 
\end{proof}

\begin{thm} \label{sbv-cmp} Let $G$ be a complete multipartite graph  on $n$ vertices where the number of singletons is 
  $m_1$     and  the number of non-singleton odd parts is $h$.
Then $G \in \SBV$  if and only if  $n$ is even  or $m_1 \ge h-1$.
\end{thm}

\begin{proof}
For the forward direction,  
we provide an SBV coloring  of $G$ when $n$ is even, or when $m_1 \ge h-1$. 
Color  each  even part so that half the vertices are blue and the other half are red.  If  $n$ is even, there are an even number of odd parts.  
 Pair them and in each pair, color one with one extra blue  vertex and the other with one extra red vertex.   This is an SBV coloring.   If $n$ is odd, there are an odd number of odd parts. Color each of the $h$ non-singleton odd parts so that it has one extra red vertex, and color $h-1$ of the singletons blue.  This accounts for an odd number of vertices, so there must be an even number of additional singletons.  Color   half of them red and half of them blue.   One can check that this is also an SBV coloring.
 
Conversely,  suppose we have   an SBV coloring  of $G$.    If $n$ is even, the first condition of the theorem is satisfied, so we assume $n$ is odd and prove that $m_1 \ge h-1$. Let $\ghat$ be the reduced graph, so by Observation~\ref{reduction-obs}, every part of $\ghat$ is monochromatic and our coloring induces an SBV coloring of $\ghat$. Let $R$ be the set of red vertices in $\ghat$ and $B$ the set of blue vertices, and without loss of generality we may assume $|R| > |B|$.

First, suppose there is a monochromatic red part $P$ with at least three vertices and let $T$ be the remaining set of vertices in $\ghat$.   Let $s$ be the number of red vertices in $T$, thus  $|R| = s+ |P|  \ge s + 3$.  If $x \in P$, then $N(x) = T$ and there are $s$ red vertices in $N(x)$ and $s+1$ in $N[x]$.  Thus there are at least $s-1$ blue vertices in $T$ and at most $s+2$, so  and we get $s-1 \le |B| \le s+2$.  
 If $|B| = s-1$ then for a blue vertex $v$, we know $N[v]$ (and $N(v)$) contain at most $s-1$ blue vertices and at least $s+3$ red vertices, a contradiction. Thus $s \le |B|\le s+2$.  Consider a blue vertex $z$ in $T$.  If $z$ is not a singleton in $\ghat$ then $N[z]$ contains at most $s+1$ blue vertices  and at least $s+3$ red vertices, and  $N(z)$  contains at most $s$ blue vertices  and at least $s+3$ red vertices. Each of these is a contradiction since our coloring of $\ghat$ is an SBV coloring.  Thus every blue vertex in $\ghat$ is a singleton, and by Lemma~\ref{reduce-lemma}, these vertices are also singletons of $G$, so $G$ has at least $s$ singletons. The red vertices of $\ghat$ constitute at most $s+1$ parts since $P$ is one part and there are $s$  red vertices in $T$.  Therefore, $G$ has at most $s+1$ non-singleton odd parts  so $h \le s+1$.  However, $G$ has at least $s$ singleton parts arising from the blue vertices of $\ghat$, so $m_1 \ge h-1$ as desired.

Second, suppose that there is a monochromatic red part $P$ with $|P| = 2$ and again let $T$ be the remaining set of vertices in $\ghat$.  Now let $t$ be the number of blue vertices in $T$.  Thus, $|B| = t$.
 If $x \in P$, then $N(x) = T$ and both $N(x)$ and $N[x]$ have $t$ blue vertices. Thus, there are at least $t-2$ red vertices in $T$ and at most $t+1$.  Since $|P|$ is even and $|V(\ghat)|$ is odd we know $|T|$ is odd, so the number of red vertices in $T$ must have the opposite parity from the number of blue vertices in $T$ and hence $T$ contains either $t-1$ or $t+1$ red vertices.  The latter is impossible because a blue vertex in $\ghat$ would have at least $t+3$ red neighbors  and $|B| = t$.  Thus there are exactly $t-1$ red vertices in $T$ and $|R| = t+1$.   Consider a blue vertex $z$ in $\ghat$.  If $z$ is not a singleton then $N(z)$  and $N[z]$  each contain   $t+1$ red vertices  and at most $t-1$ blue vertices, a contradiction.  Thus every blue vertex in $\ghat$ is a singleton, and by Lemma~\ref{reduce-lemma}, these vertices are also singletons of $G$, so $G$ has at least $t$ singletons. The $h$ non-singleton odd  parts of $G$ arise from monochromatic red odd parts of $\ghat$.  But $\ghat$ has only $t-1$ red vertices that can be in an odd part of $G$, so $h \le t-1$, or equivalently, $t \ge h+1$.  Since there are $t$ blue singleton parts in $G$, we know  $m_1 \ge t \ge  h+1 > h-1$, as desired. 
 
It remains to consider the case in which
every red part in $\ghat$ is a singleton.  Let $x$ be a red vertex and $T$ be the remaining set of vertices in $\ghat$.   Let $s$ be the number of red vertices in $T$, thus $|R| = s+1$.  
There are $s$ red vertices in $N(x)$ and $s+1$ in $N[x]$.  Thus there are at least $s-1$ blue vertices in $T$ and at most $s+2$, so we get $s-1 \le |B| \le s+2$.  Since $|V(\ghat)|$ is odd we know $|T|$ is even, so the number of red vertices in $T$ must have the same parity as the number of blue vertices in $T$ and hence $|B|=s $ or $|B| = s+2$.  The latter is impossible since the coloring is red-heavy, thus $|R| = s+1$ and $|B| = s$.  Consider a blue vertex $z$ in $\ghat$.  If $z$ is not a singleton in $\ghat$ then $N[z]$  (and $N(z)$ ) each contain  $s+1$ red vertices and at most $s-1$ blue vertices, a contradiction since our coloring of $\ghat$ is an SBV coloring.  Thus every blue vertex in $\ghat$ is a singleton, and by Lemma~\ref{reduce-lemma}, these vertices are also singletons of $G$, so  $m_1 \ge s$. The remaining vertices of $\ghat$ are $s+1$ red singletons, so they constitute at most $s+1$  non-singleton odd parts in $G$, and by Observation~\ref{reduction-obs} there are no additional non-singleton odd parts of $G$.  Therefore,  $h \le s+1 \le m_1+1$ as desired. 
\end{proof}

The most complex of our characterization theorems in this section is for CSB graphs.
The next   lemma is helpful in  allowing us to conclude that certain parts in a CSB coloring will be monochromatic.

\begin{lem}
Let $G$ be an $n$ vertex complete multipartite graph with a CSB coloring.  If $n$ is even then every odd part is monochromatic, and if $n$ is odd then every even part is monochromatic.   

\label{multi-lem}
\end{lem}

\begin{proof}
Fix a CSB coloring of an $n$ vertex complete multipartite graph $G$. Suppose that vertices $x$ and $y$ are in the same part where $x$ is red and $y$ is blue.  Let $N(x)$ consist of $x_R$ red vertices and $x_B$ blue vertices. Note that $N(y) = N(x)$.   The coloring is $1$-balanced at $N[x]$ so $-1 \le (1+ x_R) - x_B \le 1$ and so 
 $-2 \le x_R - x_B \le 0$.  The closed neighborhood $N[y]$ consists of $x_R$ red vertices  and $1 + x_B$ blue vertices. Our coloring is $1$-balanced at $N[y]$, so $-1 \le x_R - (1+x_B) \le 1$ or equivalently
 $0 \le x_R - x_B \le 2$.  Thus, $x_R-x_B = 0$ and the coloring is $0$-balanced at $N(x)$ and $N(y)$. Thus, $|N(x)| $ is even (as is $|N(y)| )$.  If $n$ is even and $x$ is in an odd part, then  $|N(x)| $ is odd, a contradiction.  So for $n$ even, all odd parts are monochromatic.  Similarly, if $n$ is odd and $x$ is in an even part then  $|N(x)| $ is odd, a contradiction.  Hence for $n$ odd, all even parts are monochromatic.
\end{proof}

A complete multipartite graph with just one part or in which each part is a singleton, has the form $\overline{K_n}$  or $K_n$.  These graphs are   CSB if and only if $n$ is even.   The next theorem handles the remaining instances.

\begin{thm} 
Let $G$ be a complete multipartite graph  on $n$ vertices that has more than one part, and a part that is not a singleton. 
Let $m_1$ be the number of singletons, $m_2$ be the number of parts of size 2, and $h$ be the number of non-singleton odd parts.
Then $G\in \CSB$  if and only if one of the following two conditions hold:

    \begin{enumerate}
        
        \item \label{csb-01}  $n$ is even and every odd part is a singleton.  
        \item \label{csb-02} 
         $n$ is odd, all even parts have size $2$,  and $m_1 \ge h + 2m_2-1$. 
         
    
       \end{enumerate}
\label{csb-complete-multi-thm}     
\end{thm}

\begin{proof}
Let $G$ be a complete multipartite graph on $n$ vertices that has more than one part and a part that is not a singleton. 
We begin by proving the backward direction and suppose that $G$ satisfies one of the three conditions.

If $G$ satisfies (\ref{csb-01}), then $n$ is even,  and every odd part is a singleton, so there must be an even number of singletons.  
 Color half the singletons red and half of them blue.  The remaining parts each have even size, so for each one, color half the vertices red and the other half blue.  This is a CSB coloring.

Next, consider graphs $G$ satisfying (\ref{csb-02}), so $n$ is odd and any even parts have size $2$.  If $h=0$, then all odd parts are singletons and  $m_1$ must be odd.  Color   the  $2m_2$ vertices in  the size $2$ parts red and color $2m_2-1$ of the singletons blue.  An even number of singletons remain, and we color half of them red and the other half blue.  This is a CSB coloring. Otherwise, $h \ge 1$ and $m_1$ and $h$ have the opposite parity because the number of odd parts of $G$ must be odd, thus $m_1 = (h-1) + 2m_2 + 2a$  for some nonnegative integer $a$.
 Color each vertex in a size $2$ part red, and for each of the $h$ non-singleton odd parts, color  the vertices so that there is one more red than blue. Color $a$ of the singletons red and the remaining $(h-1) + 2m_2 + a$ singletons blue.  The total number of red vertices is one more than the total number of blue vertices, so the coloring in $1$-balanced at $N[x]$ when $x$ is a singleton.  The coloring is $0$-balanced at $N[y]$ for each $y$ in a part of size $2$, and  also $0$-balanced at $N(z)$ for each $z$ in a non-singleton odd part, so $1$-balanced at $N[z]$. Thus, we have a CSB coloring.

\smallskip

Conversely,  suppose that $G$ has a CSB coloring and first consider the case in which $n$ is even. By Lemma~\ref{multi-lem}, every odd part is monochromatic.
Let $P$ be an odd part of $G$ and without loss of generality we may assume $P$ is monochromatic red.  For $x \in  P$ we know $|N[x]|$ is even so the coloring is $0$-balanced at $N[x]$ and hence  $N[x]$ consists of $s +1$ red and  $s+1$ blue vertices for some $s$.    Thus  in $G$ there are a total of $s + |P|$ red vertices and $s+1$ blue vertices.  Then $N(x)$ consists of $s$ red and $s+1$ blue vertices, so there must be a part $Q$ of $G$ that is blue-heavy.  Let $Q$ consist of $q_B$ blue vertices and $q_R$ red vertices, so $q_B \ge q_R + 1$.  For a blue vertex $y \in Q$ we know $N[y]$ consists of $s + |P| - q_R$ red vertices and   $s + 2 - q_B$ blue vertices.  The coloring is CSB so the difference between these quantities is at most $1$ and we get
$|P| + (q_B - q_R) -2 \le 1$.  Since $q_B - q_R \ge 1$ we know $|P| \le 2$.  However, $P$  is an odd part, so $|P| = 1$ and we conclude that every odd part of $G$ is a singleton as desired.

 It remains to consider the case which $n$ is odd. By Theorem~\ref{sbv-cmp}, $m_1\geq h-1$, so if there are no even parts, then the statement (2) holds.   Otherwise, assume that $G$ contains at least one even part and note that by Lemma~\ref{multi-lem}, every even part is monochromatic.
Let  $P$ be an even part of $G$ and without loss of generality we may assume that $P$ is monochromatic red.  For $x \in  P$ we know $|N[x]|$ is even so the coloring is $0$-balanced at $N[x]$ and hence  $N[x]$ consists of $s +1$ red and  $s+1$ blue vertices for some $s$.    Thus  in $G$ there are a total of $s + |P|$ red vertices and $s+1$ blue vertices.   Then $N(x)$ consists of $s$ red and $s+1$ blue vertices, so there must be a part $Q$ of $G$ that is blue-heavy.  Let $Q$ consist of $q_B$ blue vertices and $q_R$ red vertices, so $q_B \ge q_R + 1$.  For a blue vertex $y \in Q$ we know $N[y]$ consists of $s + |P| - q_R$ red vertices and   $s + 2 - q_B$ blue vertices.  So as before we know this difference is at most $1$ and we get $|P| + (q_B - q_R)  \le 3$.  Since $q_B \ge q_R + 1$, and $|P|$ is even, we know $|P| = 2$ and $q_B = q_R + 1$. If there exists a red vertex $z$  in   $Q$, then $N[z]$ consists of  $s+2-q_R+1$ red vertices and $s+1-q_B$ blue vertices.  However, the difference between these quantities is   $(s+2-q_R+1)-(s+1-q_B)=
2+q_B-q_R = 3$, a contradiction since the coloring is $1$-balanced at $N[z]$. Thus, every blue-heavy part of $G$  is a blue singleton  and  the $h$ non-singleton odd parts of $G$ are all red-heavy.  Thus, $G$ consists of $m_2$ monochromatic red parts of size $2$,  $h$ red-heavy non-singleton odd parts, and $m_1$ singleton parts.  Each non-singleton odd part of $G$ has at least one more red than blue vertex and each part of size $2$ has two more red than blue vertices.  Thus, for a blue singleton $v$ in $G$ the closed neighborhood $N[v]$ has at least $h + 2m_2$ more red than blue vertices from non-singleton parts, and hence there must be at least $h+2m_2 -1$ blue singletons in $N[v]$.  Therefore, $m_1 \ge h + 2m_2 -1$, as desired.
\end{proof}

Since $2$-colorings that are PB are also OSB and CSB, we can use Theorems~\ref{osb-multipartite-thm} and  \ref{csb-complete-multi-thm}  and  Lemma~\ref{multi-lem}  in characterizing those complete multipartite graphs that are parity balanced.

\begin{thm}
A complete multipartite graph $G$ with at least two parts is in $\PB$  if and only if  $|V(G)|$ is even and every odd part is a singleton. 

\label{pb-complete-multi-thm}
\end{thm}

\begin{proof}
Let $G$ be a complete multipartite graph with at least two parts. To prove the backwards direction, we assume that $|V(G)|$ is even and each odd part is a singleton.  Thus, $G$ consists of even parts and an even number of singletons.  Color half of the singletons red and the other half blue, and for each even part, color half the vertices red and the other half blue.   Each vertex $x$ in an even part has even degree and the coloring is $0$-balanced at $N(x)$.  Each singleton vertex $y$ has odd degree and the coloring is $0$-balanced at $N[y]$. Thus, our coloring is a PB coloring and $G \in \PB$.

 Conversely, let $G$ be a complete multipartite graph with at least two parts that is in $ \PB$.
First, we show that $|V(G)|$ is even.  For a contradiction, assume $|V(G)|$ is odd.   
Since PB colorings are also OSB colorings, we use Theorem~\ref{osb-multipartite-thm} to conclude that $G$ has exactly one odd part $P$.  Let $P$ consist of $p_R$ red vertices and $p_B$ blue vertices, and without loss of generality we may assume that $p_R > p_B$.  For $x \in P$ we know  $|N(x)|$ is even, so $N(x)$ consists of $s$ red and $s$ blue vertices for some $s \ge 1$.  Since $G$ has exactly one odd part, the remaining parts are even and by Lemma~\ref{multi-lem}, each of them is monochromatic.  Let $Q$ be a monochromatic blue part, so $|Q|$ is even and $|Q| \ge 2$.  For $z \in Q$, we know $|N[z]|$ is even and consists of $s + p_B - |Q| + 1$ blue vertices and $s+p_R$ red vertices.  The coloring is PB, so it is $0$-balanced at $N[z]$ and thus these quantities are equal.  Hence $p_R - p_B = 1-|Q|$.  This is a contradiction since $p_R - p_B > 0$ and 
$1-|Q|< 0$. 

 Thus, $|V(G)|$ is even.  Our PB coloring is also a CSB coloring so we can apply Theorem~\ref{csb-complete-multi-thm} to conclude that every odd part of $G$ is a singleton.
\end{proof}

We conclude by showing that, in contrast to Theorem~\ref{arb-large-thm}, each of the balance numbers for complete multipartite graphs is at most $3$.

\begin{thm}
If $G$ is a complete multipartite graph, then  $\beta_2([G])\leq 2$, \ $\beta_2(G)\leq 2$, and   $\beta_2[G]\leq 3$. Each bound is tight.

\end{thm}

\begin{proof}
Let $G$ be a complete multipartite graph.  For each even part, color half of the vertices red and the other half blue.   Pair the odd parts and for each pair, color one so that it has one more red than blue vertex and the other so that it has one more blue than red.
There may be an odd part remaining, and, if so, color it with one more red than blue.  One can check that this coloring is $2$-balanced at $N(v)$ for every vertex $v$ and thus $\beta_2(G)\leq 2$.   By Lemma~\ref{balance-lemma} \ref{bal-lem-pt2}, we know $\beta_2([G])\leq 2$ and, by Lemma~\ref{balance-lemma} \ref{bal-lem-pt4}, we have 
$\beta_2[G] \leq 3$.

It remains to prove that these bounds are tight.  By Theorem~\ref{sbv-cmp}, we know $K_{3,3,3} \not\in \SBV$ and thus $\beta_2([K_{3,3,3}]) \ge 2$, showing that the bound  $\beta_2([G])\leq 2$  is tight.  Lemma~\ref{balance-lemma} \ref{bal-lem-pt2} implies that $\beta_2(K_{3,3,3})  \ge \beta_2([K_{3,3,3}]) $ so 
$\beta_2(K_{3,3,3}) \ge 2$ and the bound  $\beta_2(G)\leq 2$  is tight.  Finally, every vertex in $K_{3,3,3}$ has even degree, so $\beta_2[K_{3,3,3}]$ is odd by Lemma~\ref{even-deg-lem} and $\beta_2[K_{3,3,3}] > 1 $ by Theorem~\ref{csb-complete-multi-thm}.  Thus, $\beta_2[K_{3,3,3}] \ge 3 $ and the bound $\beta_2[G] \leq 3$ is tight.
\end{proof}

\section{Future Work}

In this paper, we considered coloring the vertices of a graph $G$ using $k$ colors so that for each vertex $v$ there is a balance of colors in $N(v)$ or $N[v]$.  This framework suggests several natural directions for further investigation.  One extension is to consider $\lambda$-balance at  broader  vertex neighborhoods, such as the set of the vertices within distance two.

Another direction is to vary the optimization criteria.  In our work, we minimize the maximum difference between the number of vertices in each color class within each neighborhood.  Instead, one could minimize the sum of these differences.  Alternatively, we could minimize the total number of vertices that must fail to be $0$-balanced in any $k$-coloring of a graph. 

A different variation is to consider edge-colored graphs in place of vertex-colored graphs.  In this context, the goal is to balance the number of edges of each color incident to each vertex.      

    \smallskip
    
We gratefully acknowledge the American Institute of Mathematics (AIM) and their support of research collaboration. We started working together on this topic at the AIM workshop, {\it Graph Theory: structural properties, labelings, and connections to applications}, held July 22-26, 2024.


\begin{thebibliography}{99}

\bibitem{A_arxiv} M.G. Almeida. Quasi Neighborhood Balanced Coloring of Graphs, arXiv:2512.24293 (2025).

\bibitem{APGS_arxiv} M. Almeida, R. Pawar, S. Gupta, T. Singh. Closed Neighborhood Balanced $k$-Coloring of Graphs, arXiv:2510.16666 (2025).

\bibitem{ASGP_arxiv} M.G Almeida, T. Singh, S. Gupta, R. Pawar. Neighborhood Balanced k-Coloring of Graphs, arXiv:2509.06003 (2025).

\bibitem{Coetal}  K.L. Collins, M Bowie, N.B. Fox,  B. Freyberg, J. Hook, A.M. Marr, C. McBee,  A.~Semani\v{c}ov\'{a}-Fe\v{n}ov\v{c}\'{i}kov\'a, A. Sinko, and A.N. Trenk,
\newblock Closed neighborhood balanced coloring of graphs, \emph{Graphs Combin.} {\bf 41} (2025) Article no. 88. https://doi.org/10.1007/s00373-025-02950-5

\bibitem{FM24} B.~Freyberg and A.~Marr, Neighborhood balanced colorings of graphs,
\emph{Graphs Combin.} {\bf 40} (2024) Article no. 41. https://doi.org/10.1007/s00373-024-02766-9.

\bibitem{Larson} C. Larson, personal communication. He credits the proof to Virginia Commonwealth University undergraduate, Yunus Bidav.

\bibitem{MS_arxiv} M. Minyard, M.R. Sepanski.  Neighborhood Balanced 3-Coloring, arXiv:2410.05422 (2024).

\bibitem{Stanley-EC1} R. P. Stanley, {\it Enumerative Combinatorics Volume 1}, second edition, Cambridge Stud. Adv. Math. {\bf 49}, Cambridge University Press, Cambridge 2012. 


\bibitem{W69} M.~E.~Watkins, A theorem on Tait colorings with an application to the generalized Petersen graphs,\emph{ J. Comb. Theory} {\bf 6(2)} (1969) 152-164.
\bibitem{We01} D.B.~West, \emph{Introduction to Graph Theory,} Prentice-Hall Inc., NJ, 2nd edition (2001).







\end{thebibliography}
\end{document}